\documentclass[10pt,reqno]{amsart}

\usepackage{dsfont,amsmath,amsfonts,amscd,amssymb,mathrsfs,eufrak,upgreek,enumitem,mathrsfs,colonequals}
\usepackage[usenames]{color}

\usepackage{setspace}

\setlist[enumerate]{format=\normalfont}

\newcommand{\marginparstretch}{0.6}
\let\oldmarginpar\marginpar
\renewcommand\marginpar[1]{\-\oldmarginpar[\framebox{\setstretch{\marginparstretch}\begin{minipage}{\marginparwidth}{\raggedleft\tiny #1}\end{minipage}}]{\framebox{\setstretch{\marginparstretch}\begin{minipage}{\marginparwidth}{\raggedright\tiny #1}\end{minipage}}}}

\usepackage[colorlinks]{hyperref}

\usepackage{tikz}
\usetikzlibrary{arrows,decorations.pathmorphing,decorations.pathreplacing,positioning,shapes.geometric,shapes.misc,decorations.markings,decorations.fractals,calc,patterns}
\tikzset{
        cvertex/.style={circle,draw=black,inner sep=1pt,outer sep=3pt},
        vertex/.style={circle,fill=black,inner sep=1pt,outer sep=3pt},
        DBs/.style={circle,draw=black,circle,fill=black,inner sep=0pt, minimum size=3pt},
        DB/.style={circle,draw=black,circle,fill=black,inner sep=0pt, minimum size=4pt},
         DWs/.style={circle,draw=black,circle,fill=white,inner sep=0pt, minimum size=3pt},
         DWds/.style={circle,draw=black,densely dotted,circle,fill=white,inner sep=0pt, minimum size=3pt},
        DW/.style={circle,draw=black,inner sep=0pt, minimum size=4pt},
        tvertex/.style={inner sep=1pt,font=\scriptsize},
        gap/.style={inner sep=0.5pt,fill=white},
        mid/.style={inner sep=0.5pt},
        Ggap/.style={inner sep=0.5pt,fill=green!40!black!20}}

\newcommand{\rotateRPY}[3]% roll, pitch, yaw
{   \pgfmathsetmacro{\rollangle}{#1}
    \pgfmathsetmacro{\uppitchangle}{#2}
    \pgfmathsetmacro{\yawangle}{#3}

    \pgfmathsetmacro{\newxx}{cos(\yawangle)*cos(\uppitchangle)}
    \pgfmathsetmacro{\newxy}{sin(\yawangle)*cos(\uppitchangle)}
    \pgfmathsetmacro{\newxz}{-sin(\uppitchangle)}
    \path (\newxx,\newxy,\newxz);
    \pgfgetlastxy{\nxx}{\nxy};

    \pgfmathsetmacro{\newyx}{cos(\yawangle)*sin(\uppitchangle)*sin(\rollangle)-sin(\yawangle)*cos(\rollangle)}
    \pgfmathsetmacro{\newyy}{sin(\yawangle)*sin(\uppitchangle)*sin(\rollangle)+ cos(\yawangle)*cos(\rollangle)}
    \pgfmathsetmacro{\newyz}{cos(\uppitchangle)*sin(\rollangle)}
    \path (\newyx,\newyy,\newyz);
    \pgfgetlastxy{\nyx}{\nyy};

    \pgfmathsetmacro{\newzx}{cos(\yawangle)*sin(\uppitchangle)*cos(\rollangle)+ sin(\yawangle)*sin(\rollangle)}
    \pgfmathsetmacro{\newzy}{sin(\yawangle)*sin(\uppitchangle)*cos(\rollangle)-cos(\yawangle)*sin(\rollangle)}
    \pgfmathsetmacro{\newzz}{cos(\uppitchangle)*cos(\rollangle)}
    \path (\newzx,\newzy,\newzz);
    \pgfgetlastxy{\nzx}{\nzy};
}

\tikzset{RPY/.style={x={(\nxx,\nxy)},y={(\nyx,\nyy)},z={(\nzx,\nzy)}}}

\tikzstyle{mybox} = [draw=black, fill=blue!10, very thick,
    rectangle, rounded corners, inner sep=10pt, inner ysep=20pt]
\tikzstyle{boxtitle} =[fill=blue!50, text=white,rectangle,rounded corners]

\newtheorem{thm}{Theorem}[section]
\newtheorem{prop}[thm]{Proposition}
\newtheorem{lemma}[thm]{Lemma}
\newtheorem{defin}[thm]{Definition}
\newtheorem{cor}[thm]{Corollary}

\theoremstyle{definition} 

\newtheorem{example}[thm]{Example}
\newtheorem{setup}[thm]{Setup}

\newtheorem{remark}[thm]{Remark}

\newtheorem{notation}[thm]{Notation}

\numberwithin{equation}{section}

\newcounter{enumeratenoindentcounter}

\renewcommand{\t}[1]{\textnormal{#1}}

\def\op{\mathop{\rm op}\nolimits}

\def\CM{\mathop{\rm CM}\nolimits}

\def\depth{\mathop{\rm depth}\nolimits}
\def\fl{\mathop{\sf fl}\nolimits}

\def\mod{\mathop{\rm mod}\nolimits}

\def\coh{\mathop{\rm coh}\nolimits}

\def\refl{\mathop{\rm ref}\nolimits}

\def\proj{\mathop{\rm proj}\nolimits}

\def\pd{\mathop{\rm pd}\nolimits}

\def\Hom{\mathop{\rm Hom}\nolimits}
\def\RHom{\mathop{\mathbf{R}\rm Hom}\nolimits}

\def\uprhom{\mathop{\rm {\bf R}Hom}\nolimits}

\def\End{\mathop{\rm End}\nolimits}
\def\Ext{\mathop{\rm Ext}\nolimits}
\def\Tor{\mathop{\rm Tor}\nolimits}

\def\add{\mathop{\rm add}\nolimits}

\def\Aut{\mathop{\rm Auteq}\nolimits}
\def\Autgp{\mathop{\rm Autgp}\nolimits}
\def\Spec{\mathop{\rm Spec}\nolimits}
\def\sup{\mathop{\rm sup}\nolimits}

\def\Span{\mathop{\rm Span}\nolimits}

\def\Db{\mathop{\rm{D}^b}\nolimits}
\def\Kb{\mathop{\rm{K}^b}\nolimits}

\def\flop{{\sf{F}}}

\newcommand{\con}{\mathrm{con}}

\def\tilt{\mathop{\sf tilt}\nolimits}
\def\EG{\mathop{\sf EG}\nolimits}

\def\Fac{\mathop{\sf Fac}\nolimits}
\def\Begin{\mathop{\sf Begin}\nolimits}
\def\Ends{\mathop{\sf End}\nolimits}

\def\Path{\mathop{\sf Path}\nolimits}

\def\uprhom{{\rm{\bf R}Hom}}

\newcommand\fundgp{\uppi_{\hspace{0.5pt}1}\hspace{-0.5pt}}

\newcommand{\cA}{\mathcal{A}}

\newcommand{\cC}{\mathcal{C}}
\newcommand{\cD}{\mathcal{D}}

\newcommand{\cF}{\mathcal{F}}
\newcommand{\cG}{\mathcal{G}}
\newcommand{\cH}{\mathcal{H}}

\newcommand{\cS}{\mathcal{S}}
\newcommand{\cT}{\mathcal{T}}

\newcommand{\cV}{\mathcal{V}}

\newcommand{\cX}{\mathcal{X}}
\newcommand{\cY}{\mathcal{Y}}

\def\minus{\hbox{--}\kern 1pt}
\def\plus{\hbox{+}\kern 0.5pt}

\newcommand{\Per}{{}^{0}\mathrm{Per}}

\newcommand{\relmiddle}[1]{\mathrel{}\middle#1\mathrel{}}

\usepackage{amsmath}
\usepackage{graphicx}
\usepackage{stmaryrd}
\newcommand{\vin}{\rotatebox{90}{$\in$}}
\makeatletter
\def\mapstofill@{%
   \arrowfill@{\mapstochar\relbar}\relbar\rightarrow}
\newcommand*\xmapsto[2][]{%
   \ext@arrow 0395\mapstofill@{#1}{#2}}
\makeatother

\begin{document}
\title{\textsc{Faithful Actions from Hyperplane Arrangements}}
\author{Yuki Hirano}
\address{Department of Mathematics, Kyoto University, Kitashirakawa-Oiwake-cho, Sakyo-ku, Kyoto, 606-8502, Japan.}
\email{y.hirano@math.kyoto-u.ac.jp}
\author{Michael Wemyss}
\address{Michael Wemyss, School of Mathematics and Statistics, University of Glasgow, 15 University Gardens, Glasgow, G12 8QW,
UK.}
\email{michael.wemyss@glasgow.ac.uk}
\begin{abstract}
We show that if $X$ is a smooth quasi-projective $3$-fold admitting a flopping contraction, then the fundamental group of an associated simplicial hyperplane arrangement acts faithfully on the derived category of $X$.  The main technical advance is to use torsion pairs as an efficient mechanism to track various objects under iterations of the flop functor (respectively, mutation functor).  This allows us to relate compositions of the flop functor (respectively, mutation functor) to the theory of Deligne normal form, and to give a criterion for when  a finite composition of $3$-fold flops can be understood as a tilt at a single torsion pair.  We also use this technique to give a simplified proof of Brav--Thomas \cite{BT} for Kleinian singularities. 
\end{abstract}

% 14J30  	$3$-folds
% 14E30  	Minimal model program (Mori theory, extremal rays)
% 14F05  	Sheaves, derived categories of sheaves and related constructions
% 18E30  	Derived categories, triangulated categories
% 20F36  	Braid groups; Artin groups

\thanks{The first author was a Research Fellow for the Japan Society for the Promotion of Science, and was partially supported by Grant-in-Aid 26-6240. The second author was supported by EPSRC grant~EP/K021400/2.}
\maketitle
\parindent 20pt
\parskip 0pt

\section{Introduction}
Autoequivalence groups  of the bounded derived categories $\Db(\coh X)$ of coherent sheaves of  varieties $X$ have been studied in many articles. On one hand, Bondal and Orlov \cite{BO} proved that derived categories $\Db(\coh X)$ of smooth projective varieties $X$ with $K_X$ or $-K_X$ ample have only standard autoequivalences. On the other hand, Seidel and Thomas \cite{ST} showed that if $\uppi\colon X\rightarrow \mathbb{C}^2/G$ is a minimal resolution of  a quotient singularity $\mathbb{C}^2/G$ by a finite group $G\subset SL_2(\mathbb{C})$, then the derived category $\Db(\coh X)$ has non-standard autoequivalences, called {\it spherical twists}.  Across mirror symmetry, these correspond to autoequivalences of the derived Fukaya category of a homological mirror partner $X^{\vee}$ of $X$, which arises from generalized Dehn twists along Lagrangian spheres in $X^{\vee}$ \cite{ST}. More precisely, if $C\colonequals \uppi^{-1}(0)=\bigcup_{i=1}^n C_i$ with $C_i$ irreducible, Seidel--Thomas showed that the objects $\mathcal{O}_{C_i}(-1)[1]$ induce autoequivalences $t_i\in \Aut\Db(\coh X)$, and that these together induce a group homomorphism 
\[
\begin{array}{cccc}
\uprho\colon&B_{\Gamma}                     & \xrightarrow{\qquad } &        \Aut\Db(\coh X)       \\
&\vin &                 & \vin \\
&s_i                     & \xmapsto{\hphantom{\qquad}}     & t_i
\end{array}
\]
where   $B_{\Gamma}=\langle s_1,\hdots,s_n \rangle$ is the braid group of  the dual graph of exceptional curves $\bigcup_{i=1}^n C_i$ of $\uppi$, which is a Dynkin diagram of type ADE. Seidel--Thomas showed that  $\uprho$ is injective when $\Gamma$ is of type $A$, and later Brav--Thomas showed that $\uprho$ is injective in the general case \cite{BT}. This means that there is a faithful braid group action on $\Db(\coh X)$.

Moving up one dimension, if $X\to X_{\con}$ is a flopping contraction between quasi-projective $3$-folds, where $X$ is smooth and each of the $n$ irreducible exceptional curves is individually floppable, then \cite{Pinkham,HomMMP} associates to this data a real hyperplane arrangement $\cH\subseteq\mathbb{R}^n$, as a certain intersection in an ADE root system.  The main result of \cite{DW3} is that this induces an action of the fundamental group on the derived category; more precisely there exists a group homomorphism 
\[
\upvarphi\colon \fundgp(\mathbb{C}^n\backslash \cH_\mathbb{C})\to\Aut\Db(\coh X),
\]
where $\cH_{\mathbb{C}}$ denotes the complexification of the real hyperplane arrangement $\cH\subseteq\mathbb{R}^n$.  The group $\fundgp(\mathbb{C}^n\backslash \cH_\mathbb{C})$ should be viewed as a form of \emph{pure braid group}, since in the case $\cH$ is a Coxeter arrangement, this is precisely what it is.  However, in general, $\cH$ need not be Coxeter.  Motivated by the situation of surfaces above, and also by considerations in  Bridgeland stability conditions, in this paper we prove that $\upvarphi$ is injective, that is, the action is also faithful.

In fact, we do more, and our proof also recovers the surfaces case of \cite{BT} in a much simpler way.  Some of the techniques in \cite{BT} are not suited to the $3$-fold and more general settings, and so we are forced to develop a new approach. There are four main problems:

\begin{enumerate}
\item In the $3$-fold flops setting, the action $\upvarphi$ is obtained by iterating flops.  There is no `formula' for the flop functor, unlike for spherical twists, and so tracking objects under iterated flops is much more challenging.
\item The arrangement $\cH$ need not be Coxeter, so there is no finite Weyl group from which we can use reduced expressions of elements, or Garside normal form.  
\item Higher length braid relations exist, making it harder to induct on path length. 
\item There is no explicit presentation of $\fundgp(\mathbb{C}^n\backslash \cH_\mathbb{C})$ to work with.
\end{enumerate}

It turns out that these phenomena also exist for surfaces, but we need to go to \emph{partial} resolutions of Kleinian singularities in order to see them; most work to date only considers the minimal resolution.  This is addressed further in \cite{IW9}.

To obtain our main geometric results, we restrict to the formal fibre, and manipulate tilting modules there.  The following is our main result.

\begin{thm}[\ref{main groupoid faith}, \ref{main group faith}]\label{formal fibre intro}
Suppose that $f\colon X\to\Spec \mathfrak{R}$ is a complete local $3$-fold flopping contraction, where $X$ is smooth. Then the natural functor from the Deligne groupoid $\mathds{G}_\cH$ to the natural flops groupoid is faithful.  In particular, the induced group homomorphism
\[
\upvarphi\colon \fundgp(\mathbb{C}^n\backslash \cH_\mathbb{C})\to\Aut\Db(\coh X)
\]
is injective.
\end{thm}

%\begin{thm}[\ref{main groupoid faith}, \ref{main group faith}]\label{axiom main intro}\marginpar{NEEDS CHANGED}
%Suppose that $\Lambda$ is a basic $\mathfrak{R}$-algebra, where $\mathfrak{R}$ is a complete local domain.  For a fixed indecomposable projective $P_0$, suppose that there exists a finite simplicial hyperplane arrangement $\cH$ such that
%\begin{enumerate}
%\item\label{axiom main intro 1} The assumptions \ref{main assumptions} hold, in particular the exchange graph $\EG_0\Lambda$ equals the $1$-skeleton of $\cH$.
%\item\label{axiom main intro 2} There exists $d\in\mathbb{N}$, and there exists an object $b\in\Db(\mod\End_\Lambda(T))$ for each  $T\in\tilt_0\Lambda$, such that
%\[
%[S_i,b\kern 1pt]_d\neq0\mbox{ for all }0\leq i\leq n,\quad\mbox{and}\quad[\cS,b\kern 1pt]_{\geq d+1}=0,
%\]
%where $S_0,\hdots,S_n$ are the simple $\End_\Lambda(T)$-modules, and $\cS$ is their sum.
%\end{enumerate}
% Then the natural functor from the Deligne groupoid $\mathds{G}_\cH$ to the natural tilting groupoid $\cG_\Lambda$ is faithful.  In particular, the induced group homomorphism
%\[
%\upvarphi\colon \fundgp(\mathbb{C}^n\backslash \cH_\mathbb{C})\to\Aut\Db(\mod\Lambda)
%\]
%is injective.
%\end{thm}

This immediately gives global corollaries, such as the following.

\begin{cor}[\ref{main flop result}]\label{main flop result intro}
Suppose that $f\colon X\to X_{\con}$ is a flopping contraction between quasi-projective $3$-folds, where $X$ is smooth, and all curves in the contraction $f$ are individually floppable. Then there is an injective group homomorphism 
\[
\upvarphi\colon\fundgp(\mathbb{C}^n\backslash \cH_\mathbb{C})\to \Aut\Db(\coh X).
\]
\end{cor}

There is a similar statement for when the curves are not individually floppable, but  being slightly  more technical to state, we refer the reader to \ref{main flop result arb curves}.  We also recover in Appendix~\ref{BT appendix} a simplified version of Brav--Thomas in the case of minimal resolutions of Kleinian singularities.

The main technical engine in the proof is to use the order on tilting modules to control iterations.  Our new main technical result is the following, which here we state slightly vaguely, leaving details to \S\ref{Compositions section}.

\begin{thm}[\ref{simplicial and functors ok}]\label{comp main intro}
With the assumptions in \ref{formal fibre intro}, suppose that $\upalpha\colon C\to D$ is a positive minimal path.  Then the composition of mutation functors along this path is functorially isomorphic to a single functor induced by a tilting module.
\end{thm}

Since tilting modules induce torsion pairs, this allows us to use torsion pairs to control iterations.  Applying this to $3$-fold flops, where by \cite{HomMMP} the flop functor is isomorphic to the inverse of the mutation functor, gives the following result.  The first part is implicit in \cite{DW3},  whereas the second part is new, and may be of independent interest.
 
\begin{thm}[\ref{single tilt perverse}]\label{single tilt perverse intro}
Consider two crepant resolutions 
\[
\begin{tikzpicture}
\node (X) at (-1,1) {$X$};
\node (Y) at (1,1) {$Y$};
\node (R) at (0,0) {$\Spec R$};
\draw[->] (X)--(R);
\draw[->] (Y)--(R);
\end{tikzpicture}
\]
of $\Spec R$, where $R$ is an isolated cDV singularity.
\begin{enumerate}
\item Given two minimal chains of flops connecting $X$ and $Y$, the composition of flop functors associated to each chain are functorially isomorphic.
\item Perverse sheaves on $Y$, namely $\Per(Y,R)$, can be obtained from perverse sheaves on $X$, namely $\Per(X,R)$, by a single tilt at a torsion pair.
\end{enumerate}
\end{thm}

For definitions, we refer the reader to \S\ref{geo cor sect}.

\subsection{Outline of Paper} \S\ref{prelim} contains background on hyperplane arrangements, arrangement groupoids and Deligne Normal Form.  In \S\ref{Tilting Section} we then relate this to tilting modules, under the general setting that we will consider.  So as not to disturb the flow of the paper, proofs of some of the results in \S\ref{Tilting Section} appear in Appendix~\ref{tilting appendix}.  In \S\ref{Compositions section}
 we establish in \ref{simplicial and functors ok} that compositions of tilts behave well under Deligne Normal Form, and the first consequences appear in the short \S\ref{tracking torsion pairs}.  In \S\ref{faithful proof section} we use this torsion pair viewpoint to prove the faithfulness in the complete local setting, and we give all the geometric corollaries.  In Appendix~\ref{BT appendix}, which can be read independently, we give a simple direct proof of faithfulness in the case of Kleinian singularities, to demonstrate that the torsion pair viewpoint simplifies the proof.

\subsection{Acknowledgements} 
The second author would like to thank Osamu Iyama for discussions related to the tilting theory in Appendix~\ref{tilting appendix}.  Both authors would like to thank the referee for their patience, and for their helpful comments. The majority of this work was carried out when the first author visited the Universities of Edinburgh and Glasgow during 2015/16, funded by the JSPS.  We thank the JSPS, and also the universities for their hospitality.

\subsection{Conventions}\label{conventions}  
All rings and algebras are assumed to be noetherian, and to be $k$-algebras, where $k$ is some field.   All modules are right modules, unless stated otherwise. When considering flopping contractions, the base field is assumed to be algebraically closed of characteristic zero. Throughout:
\begin{itemize}
\item For a triangulated category $\cC$, and $a,b\in\cC$, to match \cite{BT} we write 
\[
[a,b\kern 1pt]_t\colonequals \Hom_{\cC}(a,b[t]).
\] 

\item For an algebra $\Lambda$, we write $\fl \Lambda$ for the category of finite length right $\Lambda$-modules. 

\item For a noetherian ring $R$, $\CM R$ denotes the category of maximal Cohen-Macaulay $R$-modules, and $\refl R$ denotes the category of finitely generated reflexive $R$-modules.

\item For an additive category $\cC$, and an object $x\in \cC$, we write $\add x\subset \cC$ for the full subcategory consisting of direct summands of finite direct sums of $x$.

\end{itemize}

\section{Preliminaries}\label{prelim}

\subsection{Hyperplane Arrangements}

Throughout this subsection $\cH$ will denote a finite set of hyperplanes in $\mathbb{R}^n$, which we will refer to as a \emph{real hyperplane arrangement}.  Such an arrangement is called \emph{Coxeter} if it arises as the set of reflection hyperplanes of a finite real reflection group.

Recall that $\cH$ is {\it simplicial} if $\bigcap_{H\in\mathcal{H}}H=\{0\}$ and all chambers in $\mathbb{R}^n\backslash \cH$ are open simplicial cones. All Coxeter arrangements are simplicial, but the converse is false.  When $\cH$ is simplicial, we will write
\[
\cH_{\mathbb{C}}\colonequals \bigcup_{H\in\cH} H_{\mathbb{C}},
\]
where $H_{\mathbb{C}}$ denotes the complexification of $H$.  The fundamental object of interest to us is the fundamental group $\fundgp(\mathbb{C}^{n}\backslash {\cH}_\mathbb{C})$ and, as is standard, to access this combinatorially we will use the Deligne  groupoid in the next subsection.

\begin{remark}
When $\cH$ is Coxeter, it is well-known that $\fundgp(\mathbb{C}^{n}\backslash {\cH}_\mathbb{C})$ is the \emph{pure braid group} associated to the corresponding finite Coxeter group, that is, the kernel of the natural morphism from the braid group to the Weyl group.  When the arrangement is simplicial but not Coxeter, there is no such  description in terms of a kernel.
\end{remark}

When $\cH$ is a simplicial hyperplane arrangement, its $1$-skeleton is defined to be the graph with vertices corresponding to the chambers, and edges joining chambers which share a codimension one wall.

\begin{example}\label{Ex:32hyper}
As an example, consider the following hyperplane arrangement $\cH$ in $\mathbb{R}^3$, and its $1$-skeleton.  It has $7$ hyperplanes,  $32$ chambers, and is not Coxeter:
\smallskip
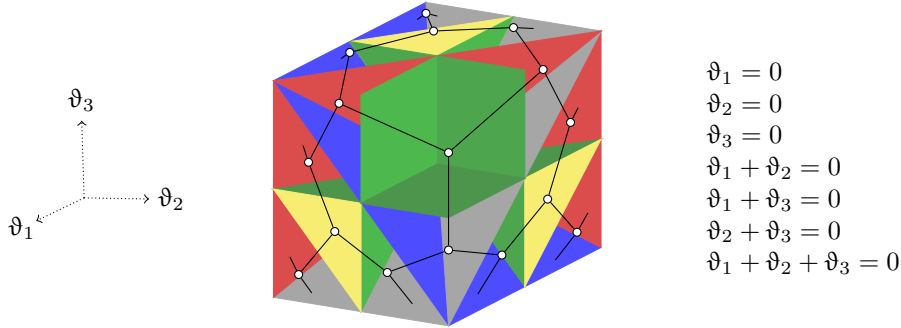
\[
\begin{array}{ccccccc}
\begin{array}{c}
\begin{tikzpicture}[scale=1]
\rotateRPY{-5}{-15}{5}
\begin{scope}[RPY]
\draw [->,densely dotted] (0,0) -- (1,0,0) node [right] {$\upvartheta_2$}; 
\draw [->,densely dotted] (0,0) -- (0,1,0) node [above] {$\upvartheta_3$}; 
\draw [->,densely dotted] (0,0) -- (0,0,1);
\node at (0,0,1.3) {$\upvartheta_1$}; 
\end{scope}
\end{tikzpicture}
\end{array}
&&
\begin{array}{c}
\begin{tikzpicture}[scale=1.2]
\rotateRPY{-0}{-20}{0}
\begin{scope}[RPY]
\filldraw[gray!70] (-1,-1,1) -- (-1,1,-1) -- (1,-1,1) -- (-1,-1,1);
\filldraw[red!80!black!70] (-1,1,1) -- (-1,-1,1) -- (0,0,0) -- cycle;
\filldraw[green!40!black!70] (-1,0,1) -- (0,0,1) -- (0,0,0)-- cycle;
\filldraw[green!60!black!70] (0,-1,1) -- (0,0,1) -- (0,0,0)-- cycle;
\filldraw[blue!70] (-1,1,1) -- (-1,1,-1) -- (1,-1,-1) -- (1,-1,1);
\filldraw[red!80!black!70] (1,1,-1) -- (1,-1,-1) -- (0,0,0) -- cycle;
\filldraw[green!40!black!70] (1,0,-1) -- (1,0,0) -- (0,0,0)-- cycle;
\filldraw[green!50!black!70] (1,-1,0) -- (1,0,0) -- (0,0,0)-- cycle;
\filldraw[yellow!95!black!70] (1,-1,0) -- (1,0,-1) -- (1,0,0)-- cycle;
\filldraw[yellow!95!black!70] (-1,0,1) -- (0,-1,1) -- (0,0,1)-- cycle;
\filldraw[gray!70] (-1,1,-1) -- (1,1,-1) -- (1,-1,1)--(-1,1,-1);
\filldraw[yellow!95!black!70] (-1,1,0) -- (0,1,-1) -- (0,0,0)--(-1,1,0);
\filldraw[green!50!black!70] (-1,1,0) -- (0,0,0) -- (0,1,0) -- cycle;
\filldraw[green!60!black!70] (0,1,-1)--(0,0,0) -- (0,1,0) -- cycle;
\filldraw[red!80!black!70] (-1,1,1) -- (1,1,-1) -- (0,0,0)-- cycle;
\filldraw[green!60!black!70] (0,1,1) -- (0,0,1) -- (0,0,0)-- (0,1,0)-- cycle;
\filldraw[green!50!black!70] (1,1,0) -- (1,0,0) -- (0,0,0)-- (0,1,0)-- cycle;
\filldraw[green!40!black!70] (1,0,1) -- (1,0,0) -- (0,0,0)-- (0,0,1)-- cycle;
\node (A) at (1,0.6,1) [DWs] {};
\node (B1) at (-0.6,0.7,0.6) [DWs] {};
\node (B2) at (-1,0.9,0) [DWs] {};
\node (C1) at (0.6,0.7,-0.7) [DWs] {};
\node (C2) at (0,0.9,-1) [DWs] {};
\node (D) at (-1,0.7,-1.1) [DWs] {};
\node (E) at (-1,0.9,-1) [DWs] {};
\node (F) at (1,-0.3,1) [DWs] {};
\node (F1) at (0.3,-0.6,1) [DWs] {};
\node (F2) at (-0.3,-0.3,1) [DWs] {};
\node (J) at (-0.8,-0.8,0.9) [DWs] {};
\node (F3) at (-0.6,0.3,1) [DWs] {};
\node (G1) at (1,-0.6,0.3) [DWs] {};
\node (G2) at (1,-0.3,-0.3) [DWs] {};
\node (H) at (0.9,-0.8,-0.8) [DWs] {};
\node (G3) at (1,0.3,-0.6) [DWs] {};
\draw (A)--(B1)--(B2)--(D)--(E);
\draw (A)--(C1)--(C2)--(D);
\draw (E) -- (-1,0.82,-1.15); 
\draw (E) -- (-1.1,0.83,-1); 
\draw (B2) -- (-1.1,0.83,0);
\draw (C2) -- (0.1,0.85,-1.15);
\draw (A)--(F)--(F1)--(F2)--(F3)--(B1);
\draw (F)--(G1)--(G2)--(G3)--(C1);
\draw (G2)--(H);
\draw (F2)--(J);
\draw (G3)--(1,0.4,-0.7);
\draw (H)--(0.85,-0.7,-1);
\draw (H)--(1,-0.9,-0.4);
\draw (G1)--(0.9,-0.9,0.5);
\draw (F1)--(0.5,-0.9,0.9);
\draw (F3)--(-0.75,0.4,0.9);
\draw (J)--(-0.9,-0.65,0.9);
\draw (J)--(-0.65,-0.95,0.9);
\end{scope}
\end{tikzpicture}
\end{array}
&&&
\begin{array}{l}
\upvartheta_1=0\\
\upvartheta_2=0\\
\upvartheta_3=0\\
\upvartheta_1+\upvartheta_2=0\\
\upvartheta_1+\upvartheta_3=0\\
\upvartheta_2+\upvartheta_3=0\\
\upvartheta_1+\upvartheta_2+\upvartheta_3=0
\end{array}
\end{array}
\]
This hyperplane arrangement appears for $cD_4$ singularities with three curves meeting at a point \cite[7.4]{HomMMP}; an explicit example of such a $cD_4$ singularity can be found in \cite[11.2.19]{CS}. 
\end{example}

\subsection{The Deligne Groupoid}
In this section we summarise some known combinatorial approaches to $\fundgp(\mathbb{C}^{n}\backslash {\cH}_\mathbb{C})$.  For more detailed references, see \cite{Paris,Paris2,Deligne}.

Recall that a groupoid is a small category $\cG$ such that for any two objects $g,h\in\cG$, the set of morphisms $\Hom(g,h)$ is non-empty and further all morphisms are invertible. We recall that a hyperplane arrangement $\cH$ in $\mathbb{R}^n$ induces a groupoid $\mathds{G}_{\cH}$ called the \emph{arrangement groupoid} (or \emph{Deligne groupoid}) of $\cH$. To define this, we first associate an oriented graph $\Gamma_{\cH}$ to the hyperplane arrangement $\cH$.

\begin{defin}
The vertices of $\Gamma_{\cH}$ are the chambers (i.e.\ the connected components) of $\mathbb{R}^n \backslash \bigcup_{H\in\cH}H$. There is an arrow $a\colon v_1\to v_2$ from chamber $v_1$ to chamber $v_2$ if the chambers are adjacent, otherwise there is no arrow. For an arrow $a\colon v_1\to v_2$, we set $s(a)\colonequals v_1$ and $t(a)\colonequals v_2$.
\end{defin}

\begin{example}\label{Ex 8 chambers 1}
Consider the following hyperplane arrangement $\cH$ in~$\mathbb{R}^2$, and its associated $\Gamma_{\cH}$.  We have labelled the arrows in $\Gamma_{\cH}$ by abuse of notation.
\vspace{-1em}
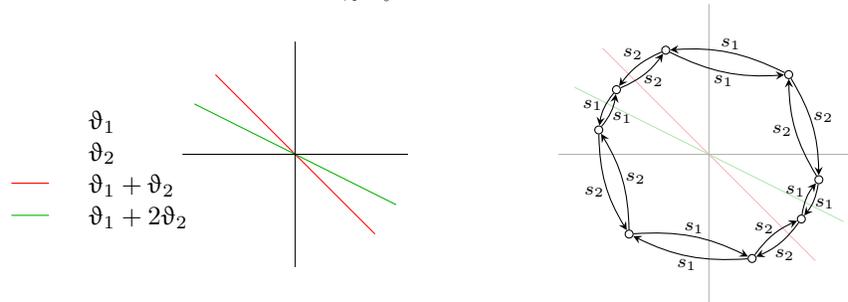
\[
\begin{array}{c}
\begin{tikzpicture}
\node at (6.5,0) {$\begin{tikzpicture}[scale=0.75,>=stealth]
\coordinate (A1) at (135:2cm);
\coordinate (A2) at (-45:2cm);
\coordinate (B1) at (153.435:2cm);
\coordinate (B2) at (-26.565:2cm);
\coordinate (C1) at (161.565:2cm);
\coordinate (C2) at (-18.435:2cm);
\draw[red] (A1) -- (A2);
\draw[green!70!black] (B1) -- (B2);
%\draw[blue] (C1) -- (C2);
\draw (-2,0)--(2,0);
\draw (0,-2)--(0,2);
\end{tikzpicture}$};
%%%%%%%%%
%%%%%%%%%
\node at (3.75,0) {
$
\begin{array}{cl}
\\
&\upvartheta_1\\
&\upvartheta_2\\
\begin{array}{c}\begin{tikzpicture}\node at (0,0){}; \draw[red] (0,0)--(0.5,0);\end{tikzpicture}\end{array}&\upvartheta_1+\upvartheta_2\\
\begin{array}{c}\begin{tikzpicture}\node at (0,0){}; \draw[green!70!black] (0,0)--(0.5,0);\end{tikzpicture}\end{array}&\upvartheta_1+2\upvartheta_2\\
\end{array}
$};
%%%%%%%%%%%%
%%%%%%%%%%%%%
\node at (12,0) 
{\begin{tikzpicture}[scale=1,bend angle=15, looseness=1,>=stealth]
\coordinate (A1) at (135:2cm);
\coordinate (A2) at (-45:2cm);
\coordinate (B1) at (153.435:2cm);
\coordinate (B2) at (-26.565:2cm);
\draw[red!30] (A1) -- (A2);
\draw[green!70!black!30] (B1) -- (B2);
\draw[black!30] (-2,0)--(2,0);
\draw[black!30] (0,-2)--(0,2);
%%nodes
\node (C+) at (45:1.5cm) [DWs] {};
\node (C1) at (112.5:1.5cm) [DWs] {};
\node (C2) at (145:1.5cm) [DWs] {};
\node (C3) at (167.5:1.5cm) [DWs] {};
\node (C-) at (225:1.5cm) [DWs] {};
\node (C4) at (-67.5:1.5cm) [DWs] {};
\node (C5) at (-35:1.5cm) [DWs] {};
\node (C6) at (-13:1.5cm) [DWs] {};
%%arrows
\draw[->, bend right]  (C+) to (C1);
\draw[->, bend right]  (C1) to (C+);
\draw[->, bend right]  (C1) to (C2);
\draw[->, bend right]  (C2) to (C1);
\draw[->, bend right]  (C2) to (C3);
\draw[->, bend right]  (C3) to (C2);
\draw[->, bend right]  (C3) to (C-);
\draw[->, bend right]  (C-) to  (C3);
\draw[<-, bend right]  (C+) to  (C6);
\draw[<-, bend right]  (C6) to  (C+);
\draw[<-, bend right]  (C6) to  (C5);
\draw[<-, bend right]  (C5) to (C6);
\draw[<-, bend right]  (C5) to  (C4);
\draw[<-, bend right]  (C4) to (C5);
\draw[<-, bend right]  (C4) to  (C-);
\draw[<-, bend right]  (C-) to (C4);
%%%labels
\node at (78.75:1cm) {$\scriptstyle s_1$};
\node at (78.75:1.5cm) {$\scriptstyle s_1$};
\node at (127:1.225cm) {$\scriptstyle s_2$};
\node at (127:1.675cm) {$\scriptstyle s_2$};
\node at (157:1.25cm) {$\scriptstyle s_1$};
\node at (157:1.675cm) {$\scriptstyle s_1$};
\node at (198:1.025cm) {$\scriptstyle s_2$};
\node at (198:1.6cm) {$\scriptstyle s_2$};
\node at (258.75:1cm) {$\scriptstyle s_1$};
\node at (258.75:1.5cm) {$\scriptstyle s_1$};
\node at (307:1.225cm) {$\scriptstyle s_2$};
\node at (307:1.675cm) {$\scriptstyle s_2$};
\node at (337:1.25cm) {$\scriptstyle s_1$};
\node at (337:1.675cm) {$\scriptstyle s_1$};
\node at (378:1.025cm) {$\scriptstyle s_2$};
\node at (378:1.6cm) {$\scriptstyle s_2$};
\end{tikzpicture}};
\end{tikzpicture}
\end{array}
\]
\end{example}

A \emph{positive path of length~$n$} in $\Gamma_{\cH}$ is defined to be a formal symbol
\[
p=a_n\circ \hdots\circ a_2\circ a_1,
\]
 whenever there exists a sequence of vertices $v_0,\hdots,v_n$ of $\Gamma_{\cH}$ and exist arrows $a_i\colon v_{i-1}\to v_i$ in $\Gamma_{\cH}$. We define $s(p)\colonequals v_0$, $t(p)\colonequals v_n$, and $\ell(p)\colonequals n$. The notation $\circ$ should remind us of composition, but we will often drop the $\circ$'s in future sections.  If 
$q=b_m\circ\hdots\circ b_2 \circ b_1$is another positive path with $t(p)=s(q)$, we consider the formal symbol
\[
q\circ p\colonequals b_m\circ\hdots\circ b_2 \circ b_1
\circ
a_n\circ \hdots\circ a_2\circ a_1,
\]
and call it the {\it composition} of $p$ and $q$. As usual, there are paths  of length zero at each vertex $v$, and by abuse of notation we will also denote the length zero path at $v$ by $v$, and identify the compositions $t(p)\circ p$ and $p\circ s(p)$ with $p$.

\begin{defin}
A positive path is called \emph{minimal} if there is no positive path in $\Gamma_{\cH}$ of smaller length, and with the same endpoints.   The positive minimal paths are called \emph{atoms}.
\end{defin}

\begin{example}\label{atoms from C plus}
In \ref{Ex 8 chambers 1}, the following are all the atoms starting in the chamber $C_+$.
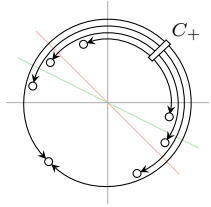
\[
\begin{tikzpicture}[scale=0.6,bend angle=15, looseness=1,>=stealth]
\coordinate (A1) at (135:2.25cm);
\coordinate (A2) at (-45:2.25cm);
\coordinate (B1) at (153.435:2.225cm);
\coordinate (B2) at (-26.565:2.25cm);
\draw[red!30] (A1) -- (A2);
\draw[green!70!black!30] (B1) -- (B2);
\draw[black!30] (-2.25,0)--(2.25,0);
\draw[black!30] (0,-2.25)--(0,2.25);
\draw[->] ([shift=(45:1.4cm)]0,0) arc (45:109.5:1.4cm);
\draw[->] ([shift=(45:1.55cm)]0,0) arc (45:142.5:1.55cm);
\draw[->] ([shift=(45:1.7cm)]0,0) arc (45:165:1.7cm);
\draw[->] ([shift=(45:1.85cm)]0,0) arc (45:222.5:1.85cm);
\draw[->] ([shift=(45:1.4cm)]0,0) arc (45:-10.5:1.4cm);
\draw[->] ([shift=(45:1.55cm)]0,0) arc (45:-32.5:1.55cm);
\draw[->] ([shift=(45:1.7cm)]0,0) arc (45:-65:1.7cm);
\draw[->] ([shift=(45:1.85cm)]0,0) arc (45:-132.5:1.85cm);
\filldraw[fill=white] (43:1.35cm) -- (43:1.9cm) -- (47:1.9cm) -- (48:1.35cm) -- cycle;
\node at (42:2.35cm) {$\scriptstyle C_+$};
\node (C1) at (112.5:1.4cm) [DWs] {};
\node (C2) at (145:1.55cm) [DWs] {};
\node (C3) at (167.5:1.7cm) [DWs] {};
\node (C-) at (225:1.85cm) [DWs] {};
\node (C4) at (-67.5:1.7cm) [DWs] {};
\node (C5) at (-35:1.55cm) [DWs] {};
\node (C6) at (-13:1.4cm) [DWs] {};
\end{tikzpicture}
\]
For each choice of start chamber, there is a similar picture.
\end{example}

Following \cite[p170]{Paris},  there is an equivalence relation $\sim$ on the set of paths in $\Gamma_{\mathcal{H}}$, defined as the smallest equivalence relation such that the following conditions are satisfied: 
\begin{enumerate}
\item If $p\sim q$, then $s(p)=s(q)$ and $t(p)=t(q)$.
\item If $p$ and $q$ are atoms with the same source and targets, then $p\sim q$.
\item If $p\sim q$, then $upr\sim uqr$ for all positive paths $u$ and $r$ satisfying $t(r)=s(p)=s(q)$, and $s(u)=t(p)=t(q)$, .
\end{enumerate}
Write $\Path\Gamma_{\cH}$ for the set of equivalence classes of positive paths in $\Gamma_{\cH}$ with respect to the equivalence relation $\sim$, and write $[p]$ for the equivalence class of a positive path $p$.
\begin{defin} 
When $\cH$ is a simplicial hyperplane arrangement, write $\mathds{G}_{\mathcal{H}}^{+}$ for the category whose objects are the vertices  in $\Gamma_{\cH}$, and whose morphisms are defined
\[
\Hom_{\mathds{G}^{+}_{\cH}}(v,u)\colonequals \{[p]\in \Path\Gamma_{\cH}\mid s(p)=v \mbox{ and } t(p)=u\}.
\]
The \emph{Deligne groupoid} (or the \emph{arrangement groupoid}) $\mathds{G}_{\mathcal{H}}$ is the groupoid defined as the groupoid completion of  $\mathds{G}_{\mathcal{H}}^{+}$, that is, adding formal inverses of all morphisms in $\mathds{G}_{\mathcal{H}}^{+}$ (see e.g.\ \cite[\S2.3.1]{Priyavrat}).
\end{defin} 
In future sections, we will abuse notation, and refer to $[\upalpha]\in\Path\Gamma_{\cH}$ simply by $\upalpha$, with the equivalence relation being implicit. The following is well-known by \cite{Deligne, Paris, Paris3, Salvetti} (see also \cite[2.1]{Paris2}), and is our main reason for considering the Deligne groupoid.

\begin{thm}\label{ver gp}
If $\cH$ is simplicial, any vertex group of the groupoid $\mathds{G}_{\cH}$ defined above is isomorphic to $\fundgp(\mathbb{C}^n\backslash \cH_\mathbb{C})$. 
\end{thm}

\subsection{Faithfulness} The faithfulness of the action of $\fundgp(\mathbb{C}^n\backslash \cH_\mathbb{C})$ on $\Db(\coh X)$ will follow from a more general faithful result on groupoids, which we briefly outline here.

\begin{defin}[{\cite[Section 1]{Deligne}}]
Assume that  $\mathcal{H}$ is simplicial. Let $v_i$ and $v_j$ be  vertices in $\Gamma_{\mathcal{H}}$, and let $C_i$ and $C_j$ be the corresponding chambers of $\mathbb{R}^n\setminus\bigcup_{H\in\mathcal{H}}H$.  Then we say that $v_j$ is {\it opposite}  to $v_i$ if there is a line $l$ in $\mathbb{R}^n$ passing through $C_i$, $C_j$, and the origin. An opposite vertex of $v$ is unique, and we denote it by $-v$. 
\end{defin}

\begin{lemma}\label{deco} 
Assume that  $\mathcal{H}$ is simplicial.  
\begin{enumerate}
\item\label{deco 1} For any atom $p$ in $\Gamma_{\mathcal{H}}$, there is an atom $p'$ such that $s(p')=-t(p)$, $t(p')=s(p)$, and  the composition $pp'$ is also an atom.
\item\label{deco 2} Let $a$ and $b$ be two atoms in $\Gamma_{\mathcal{H}}$ such that $t(a)=t(b)$. Then there are atoms $p$ and $q$ such that   $b^{-1} a= q {p}^{-1}$ in $\Hom_{\mathds{G}_{\cH}}(s(a),s(b))$.
\end{enumerate}
\end{lemma}
\begin{proof}
(1) This follows from \cite[Section 4, Corollary 2]{Paris}.\\
(2) By \eqref{deco 1}, there are atoms $p$ and $q$ such that $s(p)=s(q)=-t(a)$, and  $ap$ and $bq$ are atoms. Since the targets and sources of $ap$ and $bq$ are equal, we have $ap\sim bq$. This implies $b^{-1}\circ a= q\circ p^{-1}$ in $\Hom_{\mathds{G}_{\cH}}(s(a),s(b))$.
\end{proof}

Since $\mathds{G}_{\mathcal{H}}$ is obtained from $\mathds{G}_{\mathcal{H}}^{+}$ by adding inverses,  there is a natural functor 
\[
\iota\colon \mathds{G}_{\mathcal{H}}^{+}\longrightarrow \mathds{G}_{\mathcal{H}}.
\]
The following lemma is an easy analogue of \cite[Lemma 2.3]{BT}, and relies on the fact that $\iota$ is faithful for simplicial $\cH$.

\begin{lemma}\label{faith to pos}
Assume that $\mathcal{H}$ is simplicial, and let $F\colon\mathds{G}_{\mathcal{H}}\rightarrow \mathcal{G}$ be a functor between groupoids. Then  $F$ is faithful if and only if $F\circ \iota\colon\mathds{G}_{\mathcal{H}}^{+}\rightarrow \mathcal{G}$ is faithful.
\end{lemma} 
\begin{proof}
By \cite{Deligne}, $\iota\colon\mathds{G}_{\mathcal{H}}^{+}\longrightarrow \mathds{G}_{\mathcal{H}}$ is faithful. Thus it immediately follows that if  $F$ is faithful, so is $F\circ \iota\colon \mathds{G}_{\mathcal{H}}^{+}\rightarrow \mathcal{G}$.

For the other direction, assume that $F\circ \iota\colon\mathds{G}_{\mathcal{H}}^{+}\rightarrow \mathcal{G}$ is faithful, and let $p,q\in{\rm Hom}_{\mathds{G}_{\mathcal{H}}}(v,w)$ be morphisms. It is enough to show that if $F(p)=F(q)$ then $p=q$. 
At first, we consider the case when $v=w$. In this case, it is enough to show that, if $F(p)={\rm id}_{F(v)}$, then $p={\rm id}_{v}$.  By repeated use of \ref{deco}\eqref{deco 2}, there are positive paths $p_1$ and $p_2$ such that $p= p_1\circ{p_2}^{-1}$. Since we have  $F(p_1)=F(p_2)$ and $F\circ \iota$ is faithful, necessarily  $p_1= p_2$ and so $p= p_1\circ{p_2}^{-1}= {\rm id}_{v}$.
Next, we consider the general case when $F(p)=F(q)$. Then we have  $F(pq^{-1})={\rm id}_{F(v_j)}$. By the above argument, we see that $pq^{-1}= {\rm id}_{v}$, and thus $p=q$.
\end{proof}

\begin{cor}\label{group is faithful}
Assume that $\mathcal{H}$ is simplicial, $F\colon\mathds{G}_{\mathcal{H}}\rightarrow \mathcal{G}$ is a functor between groupoids, and for any chamber $C$ write $\Autgp(FC)\colonequals \Hom_{\mathcal{G}}(FC,FC)$.  If $F$ is a faithful functor, then there is an injective group homomorphism
\[
\fundgp(\mathbb{C}^n\backslash \cH_\mathbb{C})\to\Autgp(FC).
\]
\end{cor}
\begin{proof}
If $F$ is faithful, the induced group homomorphism $F\colon \Hom_{\mathds{G}_{\mathcal{H}}}(C,C)\rightarrow \Autgp(FC)$ is injective for any chamber $C\in\mathds{G}_{\mathcal{H}}$. Since $\Hom_{\mathds{G}_{\mathcal{H}}}(C,C)$ is isomorphic to $\fundgp(\mathbb{C}^n\backslash \cH_\mathbb{C})$ by \ref{ver gp}, the result holds.
\end{proof}

\subsection{Deligne normal form}\label{DNF section}
By \ref{faith to pos} and \ref{group is faithful} our problem will reduce to proving the faithfulness of a \emph{positive part} of a groupoid action.  This is a significant reduction in complexity, since every positive path has a \emph{Deligne normal form}, which we recall here.  This normal form replaces the Garside normal form in \cite{BT}, which is only defined for Coxeter arrangements.  The proof of faithfulness will simply induct on the number of factors of this normal form.

For positive paths $p,q\in \Gamma_{\cH}$ with $s(p)=s(q)$, we say that $p$ {\it begins} with $q$ if there exists a positive path $r$ such that  $s(r)=t(q)$, $t(r)=t(p)$ and $p\sim rq$. For a positive path $p$, write $\Begin(p)$ for the set of all atoms with which $p$ begins.  Similarly, we can consider the set of atoms with which $p$ ends, which is defined in the analogous way, and we denote this set by $\Ends(p)$.

\begin{defin}
For any path $p\in\Gamma_{\cH}$, by \cite[2.2]{Paris2} (or \cite{Deligne}), there exists a unique (up to equivalence) atom $\upalpha_1$ such that $\Begin(p)=\Begin(\upalpha_1)$. Then, in particular,  $p$ begins with $\upalpha_1$, and  so there is a positive path $\upbeta$ with $s(\upbeta)=t(\upalpha_1)$ and $t(\upbeta)=t(p)$ such that 
\[
p\sim \upbeta\circ\upalpha_1.
\] 
  Continuing this process with $\upbeta$, we decompose $p$ into atoms
\[
p\sim\upalpha_n\circ\hdots\circ \upalpha_2\circ \upalpha_1,
\]
which we refer to as the {\rm Deligne normal form of $p$}.
\end{defin}

The following lemma is convenient, and is well known \cite[Lemma 4.2]{Paris}.
\begin{lemma}\label{2.14}
If $p\in\Gamma_{\cH}$, then $p$ is an atom if and only if $p$ does not cross any hyperplane twice.
\end{lemma}

\begin{example}\label{Ex 8 chambers 2}
Continuing the example and notation in \ref{Ex 8 chambers 1}, dropping the composition symbol $\circ$, the path $p=s_2 s_1 s_2 s_1  s_2 s_2 s_1 s_1 s_2 s_1$ satisfies $\Begin(p)=\Begin(s_2 s_1 s_2 s_1)$ since
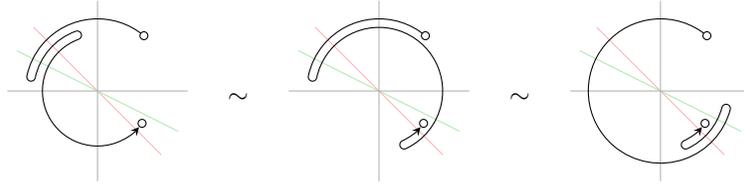
\[
\begin{array}{ccccc}
\begin{array}{c}
\begin{tikzpicture}[scale=0.6,bend angle=15, looseness=1,>=stealth]
\coordinate (A1) at (135:2cm);
\coordinate (A2) at (-45:2cm);
\coordinate (B1) at (153.435:2cm);
\coordinate (B2) at (-26.565:2cm);
\draw[red!30] (A1) -- (A2);
\draw[green!70!black!30] (B1) -- (B2);
\draw[black!30] (-2,0)--(2,0);
\draw[black!30] (0,-2)--(0,2);
\draw ([shift=(50:1.6cm)]0,0) arc (50:168:1.6cm);
\draw (168:1.6cm) arc (168:348:0.095cm);
\draw ([shift=(168:1.41cm)]0,0) arc (168:110:1.41cm);
\draw (110:1.41cm) arc (110:-70:0.095cm);
\draw[->] ([shift=(110:1.22cm)]0,0) arc (110:319:1.22cm);
\node at (50:1.6cm) [DWs] {};
\node at (-36:1.22cm) [DWs] {};
\end{tikzpicture}
\end{array}
&
\sim
&
\begin{array}{c}
\begin{tikzpicture}[scale=0.6,bend angle=15, looseness=1,>=stealth]
\coordinate (A1) at (135:2cm);
\coordinate (A2) at (-45:2cm);
\coordinate (B1) at (153.435:2cm);
\coordinate (B2) at (-26.565:2cm);
\draw[red!30] (A1) -- (A2);
\draw[green!70!black!30] (B1) -- (B2);3
\draw[black!30] (-2,0)--(2,0);
\draw[black!30] (0,-2)--(0,2);
\draw ([shift=(50:1.6cm)]0,0) arc (50:168:1.6cm);
\draw (168:1.6cm) arc (168:348:0.095cm);
\draw ([shift=(168:1.41cm)]0,0) arc (168:-65:1.41cm);
\draw (-65:1.41cm) arc (-65:-250:0.095cm);
\draw[->] ([shift=(-65:1.22cm)]0,0) arc (-65:-41:1.22cm);
\node at (50:1.6cm) [DWs] {};
\node at (-36:1.22cm) [DWs] {};
\end{tikzpicture}
\end{array}
&
\sim
&
\begin{array}{c}
\begin{tikzpicture}[scale=0.6,bend angle=15, looseness=1,>=stealth]
\coordinate (A1) at (135:2cm);
\coordinate (A2) at (-45:2cm);
\coordinate (B1) at (153.435:2cm);
\coordinate (B2) at (-26.565:2cm);
\draw[red!30] (A1) -- (A2);
\draw[green!70!black!30] (B1) -- (B2);3
\draw[black!30] (-2,0)--(2,0);
\draw[black!30] (0,-2)--(0,2);
\draw ([shift=(50:1.6cm)]0,0) arc (50:345:1.6cm);
\draw (-15:1.6cm) arc (-15:170:0.095cm);
\draw ([shift=(-15:1.41cm)]0,0) arc (-15:-65:1.41cm);
\draw (-65:1.41cm) arc (-65:-250:0.095cm);
\draw[->] ([shift=(-65:1.22cm)]0,0) arc (-65:-41:1.22cm);
\node at (50:1.6cm) [DWs] {};
\node at (-36:1.22cm) [DWs] {};
\end{tikzpicture}
\end{array}
\end{array}
\]
Continuing in this way, $p$ has Deligne normal form $s_2 (s_2 s_1 ) (s_1 s_2 s_1) (s_2 s_1 s_2 s_1)$.
\end{example}

\section{The Tilting Order and Chambers} \label{Tilting Section}

Our strategy to prove faithfulness of the action in the flops setting is to exploit the partial order on tilting modules, due to Riedtmann--Schofield and Happel--Unger \cite{RS,HU}.  In the case of minimal resolutions of Kleinian singularities, we can bypass this step by simply appealing to \cite[\S6]{IR}, and so for the proof of faithfulness in this case, the reader can skip immediately to Appendix~\ref{BT appendix}.

\subsection{Tilting Modules and Mutation}\label{tilting mutation}
Recall first that for an algebra $A$ such that the category $\mod A$ of finitely generated $A$-modules is Krull--Schmidt,  $M\in\mod A$ is called {\it basic} if there is no repetition in its Krull--Schmidt decomposition  into indecomposable $A$-modules, and the algebra $A$ is called {\it basic} if it is basic as an $A$-module.   

Throughout this section, $\Lambda$ is a basic $\mathfrak{R}$-algebra, where $\mathfrak{R}$ is a complete local domain.  Note by \cite[p566]{Swan}, for such rings the category $\mod\Lambda$ is Krull--Schmidt.    In our geometric settings later, such $\Lambda$ appear when we work on the formal fibre. 
\begin{defin}\label{tilt def}
$T\in\mod\Lambda$ is a \emph{classical tilting module} if the following conditions hold.
\begin{enumerate}
\item\label{tilt def 1} $\pd_\Lambda T\leq 1$. 
\item\label{tilt def 2} $\Ext^1_\Lambda(T,T)=0$. 
\item\label{tilt def 3} There exists a short exact sequence $0\to \Lambda\to T_1\to T_2\to 0$, with each $T_i\in\add T$.
\end{enumerate}
We write $\tilt \Lambda$ for the set of basic classical tilting $\Lambda$-modules.
\end{defin}

We shall refer to classical tilting modules simply as \emph{tilting modules}, with it being implicit that $\pd_\Lambda T\leq 1$.  
When $T$ is a tilting module, we write $\Fac T$ for the full subcategory of $\mod\Lambda$ consisting of those modules $Y$ such that there exists a surjection $T'\twoheadrightarrow Y$ with $T'\in\add T$.  It is known, and easy to prove from \ref{tilt def}\eqref{tilt def 3}, that
\begin{eqnarray}
\Fac T=\{  X\in\mod\Lambda\mid \Ext^1_\Lambda(T,X)=0\},\label{Fac T eq 1}
\end{eqnarray}
so in particular for any $X\in\Fac T$ there is an exact sequence
\[
0\to Y\to T^\prime\to X\to 0
\]
with $Y\in\Fac T$ and $T^\prime\in\add T$.  It follows immediately that 
\begin{equation}
\add T=\{  X\in\Fac T\mid \Ext^1_\Lambda(X,\Fac T)=0\}.\label{Fac T eq 3}
\end{equation}

The set $\tilt\Lambda$ carries the natural structure of a partially ordered set.
\begin{notation}\label{tilting order}
Let $T,U\in\tilt\Lambda$. We write $T\ge U$ if $\Ext^1_\Lambda(T,U)=0$, or equivalently by \eqref{Fac T eq 1}, if $U\in\Fac T$.  We write $T > U$ if $T\ge U$ and $\lnot(U\ge T)$.
\end{notation}
\noindent
It is immediate from \eqref{Fac T eq 3} and the Krull--Schmidt property that if $T,U\in\tilt\Lambda$ with $T\geq U\geq T$, then $T\cong U$. We remark that  $T\geq U$ if and only if $\Fac T\supseteq \Fac U$, and that $\Lambda\in\tilt\Lambda$ is the greatest element with respect to $\geq$.

Another key property of the set $\tilt\Lambda$ is that it admits an operation called \emph{mutation}. For $T\in \tilt\Lambda$, and an indecomposable direct summand $T_i$ of $T$,
there exists at most one basic tilting $\Lambda$-module $\upnu_iT=(T/T_i)\oplus U_i$ such that $T_i \ncong U_i$ (c.f.\ \cite{RS}).
The module $\upnu_iT$ is called a \emph{tilting mutation} of $T$, and in general it may or may not exist.   As is standard, mutation is encoded in the \emph{exchange graph} of $\tilt\Lambda$.

\begin{notation}
We write $\EG(\Lambda)$ for the \emph{exchange graph}, where vertices are elements of $\tilt \Lambda$, and we draw an edge between $T$ and $\upnu_iT$ for all $T$ and $i$ such that $\upnu_iT$ exists.  Further, for a fixed projective $P$,  let $\EG_P(\Lambda)$ denote the full subgraph of the exchange graph of $\Lambda$ consisting of those vertices that contain $P$ as a summand.
\end{notation}

\subsection{Chambers Associated to Tilting Modules}\label{subsection tilting chambers}
To functorially control compositions of tilting mutations requires chambers, which we now describe.  We first fix notation.  Let $\Lambda$ be a basic $\mathfrak{R}$-algebra, where $\mathfrak{R}$ is a complete local domain, and write $K_0\colonequals K_0(\Kb(\proj\Lambda))$. It is well known that
\begin{equation}
K_0\cong\mathbb{Z}^{n+1}\label{fix K theory basis}
\end{equation}
since every $P\in\proj\Lambda$ can be uniquely written as a direct sum of indecomposable projectives $P_0^{\oplus a_0}\oplus\hdots\oplus P_n^{\oplus a_n}$ for some $a_i$.  In what follows, we will fix the $\mathbb{Z}$-basis of $K_0$ given by \eqref{fix K theory basis}, namely $\{ \mathbf{e}_0,\hdots ,\mathbf{e}_n\}$ where $\mathbf{e}_i$ is the class of $P_i$ in $K_0$.

We now fix a projective, which by convention will be $P_0$,  and we will primarily be interested in $\EG_0(\Lambda)\colonequals \EG_{P_0}(\Lambda)$, and its vertex set $\tilt_0(\Lambda)$ consisting of all tilting $\Lambda$-modules that contain $P_0$ as a summand.  For this purpose, consider the  following factor $\mathbb{R}$-vector space of $K_0\otimes_{\mathbb{Z}}\mathbb{R}\cong\mathbb{R}^{n+1}$ given by
\[
\Uptheta_\Lambda\colonequals (K_0\otimes_{\mathbb{Z}}\mathbb{R})/\Span\{ \mathbf{e}_0\}\cong\mathbb{R}^n.
\]  
By abuse of notation, we write $\{[P_1],\hdots,[P_n]\}$ for the $\mathbb{R}$-basis of $\Uptheta_\Lambda$ induced by \eqref{fix K theory basis}, with it being implicit that the $[-]$ notation works modulo $\Span\{ \mathbf{e}_0\}$.  From this, we define 
\[
C_+\colonequals 
\left\{ \sum_{i=1}^n \upvartheta_i[P_i] 
\,\relmiddle| \, \upvartheta_i> 0\mbox{ for all }1\leq i\leq n 
\right\}
\subseteq\Uptheta_\Lambda.
\]
For $T\in\tilt_0\Lambda$, write $T=T_0\oplus T_1\oplus \hdots\oplus T_n$ where by convention $P_0=T_0$, and  consider 
\begin{align}
C_T&\colonequals 
\left\{ \sum_{i=1}^n\upvartheta_i [T_i] 
\,\relmiddle| \,
 \upvartheta_i> 0\mbox{ for all }1\leq i\leq n
\right\}\subseteq\Uptheta_\Lambda.\label{C T def}
\end{align}
It is clear from the definition that $C_\Lambda=C_+$.

 The following is elementary, and is very similar to the arguments of \cite{Hille,DIJ}.  Since the setting here does not involve Hom-finite categories, we give the proof in Appendix~\ref{tilting appendix}. 

\begin{lemma}\label{easy one way appA text}
Suppose that $\Lambda$ is a basic $\mathfrak{R}$-algebra, where $\mathfrak{R}$ is a complete local domain.  If $T,U\in\tilt_0\Lambda$ are related by a mutation at an indecomposable summand, then $C_T$ and $C_U$ do not overlap, and are separated by a codimension one wall.
\end{lemma}

 It is the following that will allow us to control iterations, as it relates the combinatorics of chamber structures to the homological property of the tilting order.  The result seems to be folkflore; for lack of a suitable reference, and since we are working slightly more generally than usual,  we give the proof in \ref{DIJ lemmaB app} in Appendix~\ref{tilting appendix}.

\begin{thm}\label{DIJ lemmaB}
Suppose that $\Lambda$ is a basic $\mathfrak{R}$-algebra, where $\mathfrak{R}$ is a complete local domain.  Suppose that $T,U\in\tilt_0\Lambda$ are related  by a mutation at an indecomposable summand, so by \ref{easy one way appA text} $C_T$ and $C_U$ are separated by $H$.  Suppose that $[\Lambda]\notin H$.  Then $T>U$ iff $C_T$ lies on the same side of $H$ as $[\Lambda]$.
\end{thm}

%
%The following is our main technical tilting lemma, which follows easily from \ref{DIJ lemma text} using standard facts about tilting modules.   This result will allow us to reduce the proof of faithfulness to a simple torsion pair calculation.  
%
%\begin{thm}\label{mut funct iso general}
%Suppose that  GENERAL NOTATION HERE!!
%\[
%\upalpha=\quad C\xrightarrow{s_{i_1}}C_2\to\hdots\to C_{m}\xrightarrow{s_{i_m}}D.
%\]
%\begin{enumerate}
%\item\label{mut funct iso 1} 
%As tilting $\Lambda_C$-modules, $\Lambda_C>\upnu_{i_1}\Lambda_{C}>\upnu_{i_2i_1}\Lambda_{C}>\hdots>\upnu_{i_m\hdots i_1}\Lambda_{C}=\upnu_{\!\upalpha}\Lambda_C$.
%\item\label{mut funct iso 2}  There is a functorial isomorphism
%\begin{eqnarray}
%t_{{i_m}}\circ\hdots\circ t_{{i_1}}\cong  \RHom_{\Lambda_C}(\upnu_{\!\upalpha}\Lambda_C,-)\label{funct iso torsion}
%\end{eqnarray}
%with $\upnu_{\!\upalpha}\Lambda_C\in\tilt\Lambda_C$. \end{enumerate}
%\end{thm}
%\begin{proof}
%(1) Since $\upalpha$ is an atom, for all $1\leq t\leq m$, the composition $s_{i_t}\circ \hdots\circ s_{i_1}$ is also an atom.  Simply applying \ref{DIJ lemma text} repeatedly to $\Lambda_C$, at each stage using \ref{2.14},  gives the result.\\
%(2) This follows immediately from \eqref{mut funct iso 1}  and standard isomorphisms (see e.g.\ \ref{big to small cor}).
%\end{proof}
%

\section{Compositions of Mutations and Flops}\label{Compositions section}

In this section we will describe compositions of mutation functors, respectively flop functors, under Deligne normal form.  This, and more generally the proof of faithfulness of the group action, will be reduced to the formal fibre, and so for much of the paper we will work under the following  setup.

\begin{setup}\label{flops setup}
Suppose that $f\colon U\to\Spec \mathfrak{R}$ is a complete local $3$-fold flopping contraction, where $U$ is smooth. 
\end{setup}

It is well-known \cite[3.2.8]{VdB1d} that in this setting $\Db(\coh U)$ admits a tilting bundle $\mathcal{V}$ generated by global sections, which after setting $M\colonequals f_*\cV$, induces an equivalence $\Db(\coh U)\cong\Db(\mod\End_{\mathfrak{R}}(M))$.  The algebra $\End_{\mathfrak{R}}(M)$ contains $\Hom_{\mathfrak{R}}(M,\mathfrak{R})$ as a summand, and in the following we fix $P_0\colonequals \Hom_{\mathfrak{R}}(M,\mathfrak{R})$, so that $\tilt_0(\Lambda)$ consists of those tilting $\Lambda$-modules 
containing $\Hom_{\mathfrak{R}}(M,\mathfrak{R})$ as a summand.

\subsection{CT objects and Simple Wall Crossings} Under the above flops setup, $M\in\CM\mathfrak{R}$ and $\End_{\mathfrak{R}}(M)$ is a NCCR \cite[3.2.9, 3.2.10]{VdB1d}.  It follows \cite[5.4]{IW4}  that $M$ is a \emph{cluster tilting} (=CT) object of $\CM\mathfrak{R}$, namely there are equalities
\[
\add M=\{ X\in\CM\mathfrak{R}\mid \Ext^1_{\mathfrak{R}}(X,M)=0\}
=\{ Y\in\CM\mathfrak{R}\mid \Ext^1_{\mathfrak{R}}(M,Y)=0\}.
\]  
We can, and will, assume that $M$ is basic. The class of basic cluster tilting objects carries an operation of \emph{mutation}, which involves picking an indecomposable summand $M_i$ of a CT module $M$, and uniquely replacing it with a different indecomposable summand whilst remaining CT; the resulting module will be denoted $\upnu_iM$. 

By the three-dimensional Auslander--McKay correspondence \cite[6.9]{HomMMP}, the number of CT $\mathfrak{R}$-modules is equal to the number of chambers of some simplicial hyperplane arrangement, described in detail in \cite[5.24, 5.25]{HomMMP}, and furthermore crossing a codimension one wall (henceforth a \emph{simple wall crossing}) corresponds to mutating an indecomposable summand of the associated CT module.  Consequently, the  $1$-skeleton of the arrangement equals the exchange graph of CT $\mathfrak{R}$-modules.  

Under the setup of \ref{flops setup}, to fix notation we will write $\cH_\Lambda$ for the simplicial hyperplane arrangement associated to $f$, set $M\colonequals f_*\cV$, which will correspond to the chamber $C_+$, and fix $\Lambda\colonequals\End_{\mathfrak{R}}(M)$.  

\begin{example}\label{KatzD4}
There exists \cite{Katz} a $cD_4$ flop with the following simplicial hyperplane arrangement.  Under the Auslander--McKay correspondence, the following picture illustrates the exchange graph of CT objects, where $\upnu_{i_2i_1}\colonequals \upnu_{i_2}\upnu_{i_1}$ etc.   
\[
\begin{tikzpicture}
\node (A) at (0,0) 
{\begin{tikzpicture}[scale=1.25,>=stealth,bend angle=20]
\coordinate (A1) at (135:2cm);
\coordinate (A2) at (-45:2cm);
\coordinate (B1) at (153.435:2cm);
\coordinate (B2) at (-26.565:2cm);
\coordinate (C1) at (161.565:2cm);
\coordinate (C2) at (-18.435:2cm);
\draw[red!30] (A1) -- (A2);
\draw[green!70!black!30] (B1) -- (B2);
%\draw[blue] (C1) -- (C2);
\draw[black!30] (-2,0)--(2,0);
\draw[black!30] (0,-1.75)--(0,1.75);
%%nodes
\node (C+) at (45:1.5cm) {$\scriptstyle M$};
\node (C1) at (112.5:1.5cm) {$\scriptstyle \upnu_1M$};
\node (C2) at (145:1.5cm) {$\scriptstyle \upnu_{21}M$};
\node (C3) at (167.5:1.5cm) {$\scriptstyle \upnu_{121}M$};
\node (C-) at (225:1.5cm) {$\scriptstyle \upnu_{1212}M$};
\node (C4) at (-67.5:1.5cm) {$\scriptstyle \upnu_{212}M$};
\node (C5) at (-35:1.5cm) {$\scriptstyle \upnu_{12}M$};
\node (C6) at (-13:1.5cm) {$\scriptstyle \upnu_2M$};
%%edges
\draw[bend right]  (C+) to (C1);
\draw  (C1) to (C2);
\draw  (C3) to (C2);
\draw[bend left]  (C-) to  (C3);
\draw[bend left]  (C+) to  (C6);
\draw (C6) to  (C5);
\draw (C5) to  (C4);
\draw[bend left]  (C4) to  (C-);
\end{tikzpicture}};
%%%%%
\end{tikzpicture}
\]
\end{example}

Thus, under the setup of \ref{flops setup}, via \cite[5.24, 5.25]{HomMMP} every  chamber $C$ in $\cH_\Lambda$ has an associated CT $\mathfrak{R}$-module $N_C$ say, and thus an associated derived category $\Db(\mod\Lambda_C)$, where $\Lambda_C\colonequals \End_{\mathfrak{R}}(N_C)$.  There are natural equivalences between these categories, as follows.

 \begin{notation}\label{t for length one}
Suppose that $\upalpha\colon C\to D$ is an atom in $\Gamma_{\cH_\Lambda}$.  Then by \cite[4.17]{IW4} $T_{CD}\colonequals\Hom_{\mathfrak{R}}(N_C,N_D)$ is a tilting bimodule from $\Lambda_C$ to $\Lambda_D$, and we consider the equivalence 
\[
 \Db(\mod\Lambda_C)\xrightarrow{t_{\upalpha}\colonequals\RHom_{\Lambda_C}(T_{CD},-)}\Db(\mod\Lambda_D).
\]
When $\upalpha$ is a simple wall crossing, mutating the $i$th summand of $N_C$ say, we will write
\[
t_i\colonequals \RHom_{\Lambda_C}(T_{CD},-),
\] 
and refer to $t_i$ as the \emph{mutation functor}.
\end{notation}

\begin{remark}
By \cite[4.2]{HomMMP}, the functor $t_i$ is functorially isomorphic to the inverse of the flop functor, flopping a single curve $C_i$.   
\end{remark}
 
It is known \cite[3.22]{DW3} that the mutation functors $t_i$  form a representation of the Deligne groupoid, and thus they alone are enough to induce the action of the fundamental group. However, it is the existence of the additional functors $t_\upalpha$ for every atom $\upalpha$ that will allow us to control this action, and prove faithfulness in this paper.

\begin{example}\label{wall crossing functors}
Continuing the example \ref{KatzD4}, setting $\Lambda_I\colonequals \End_{\mathfrak{R}}(\upnu_IM)$, the mutation functors $t_i$ are as follows:
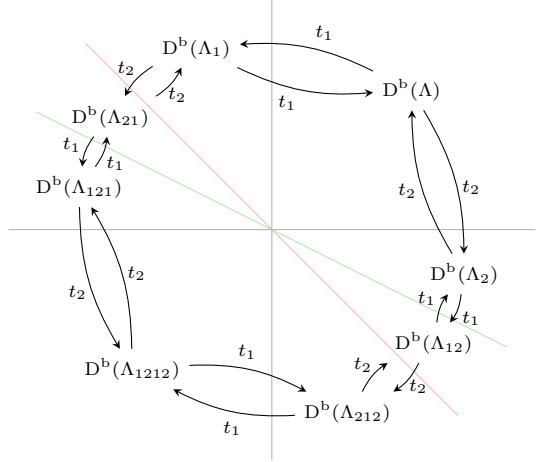
\[
\begin{array}{c}
\begin{tikzpicture}
\node at (0,0) 
{\begin{tikzpicture}[scale=1.75,bend angle=15, looseness=1,>=stealth]
\coordinate (A1) at (135:2cm);
\coordinate (A2) at (-45:2cm);
\coordinate (B1) at (153.435:2cm);
\coordinate (B2) at (-26.565:2cm);
\draw[red!30] (A1) -- (A2);
\draw[green!70!black!30] (B1) -- (B2);
\draw[black!30] (-2,0)--(2,0);
\draw[black!30] (0,-1.75)--(0,1.75);
%%nodes
\node (C+) at (45:1.5cm) {$\scriptstyle\Db(\Lambda)$};
\node (C1) at (112.5:1.5cm) {$\scriptstyle\Db(\Lambda_1)$};
\node (C2) at (145:1.5cm) {$\scriptstyle\Db(\Lambda_{21})$};
\node (C3) at (167.5:1.5cm) {$\scriptstyle\Db(\Lambda_{121})$};
\node (C-) at (225:1.5cm) {$\scriptstyle\Db(\Lambda_{1212})$};
\node (C4) at (-67.5:1.5cm) {$\scriptstyle\Db(\Lambda_{212})$};
\node (C5) at (-35:1.5cm) {$\scriptstyle\Db(\Lambda_{12})$};
\node (C6) at (-13:1.5cm) {$\scriptstyle\Db(\Lambda_{2})$};
%%arrows
\draw[->, bend right]  (C+) to (C1);
\draw[->, bend right]  (C1) to (C+);
\draw[->, bend right]  (C1) to (C2);
\draw[->, bend right]  (C2) to (C1);
\draw[->, bend right]  (C2) to (C3);
\draw[->, bend right]  (C3) to (C2);
\draw[->, bend right]  (C3) to (C-);
\draw[->, bend right]  (C-) to  (C3);
\draw[<-, bend right]  (C+) to  (C6);
\draw[<-, bend right]  (C6) to  (C+);
\draw[<-, bend right]  (C6) to  (C5);
\draw[<-, bend right]  (C5) to (C6);
\draw[<-, bend right]  (C5) to  (C4);
\draw[<-, bend right]  (C4) to (C5);
\draw[<-, bend right]  (C4) to  (C-);
\draw[<-, bend right]  (C-) to (C4);
%%%labels
\node at (83:0.97cm) {$\scriptstyle t_1$};
\node at (75:1.55cm) {$\scriptstyle t_1$};
\node at (125:1.25cm) {$\scriptstyle t_2$};
\node at (132:1.65cm) {$\scriptstyle t_2$};
\node at (157:1.29cm) {$\scriptstyle t_1$};
\node at (157:1.65cm) {$\scriptstyle t_1$};
\node at (198:1.08cm) {$\scriptstyle t_2$};
\node at (198:1.55cm) {$\scriptstyle t_2$};
\node at (258.75:0.95cm) {$\scriptstyle t_1$};
\node at (258.75:1.54cm) {$\scriptstyle t_1$};
\node at (304:1.24cm) {$\scriptstyle t_2$};
\node at (313:1.63cm) {$\scriptstyle t_2$};
\node at (336:1.29cm) {$\scriptstyle t_1$};
\node at (336:1.66cm) {$\scriptstyle t_1$};
\node at (376:1.07cm) {$\scriptstyle t_2$};
\node at (372:1.55cm) {$\scriptstyle t_2$};
\end{tikzpicture}};
\end{tikzpicture}
\end{array}
\] 
There are more direct functors, for all atoms.  As in \ref{atoms from C plus}, for those out of $C_+$ these are 
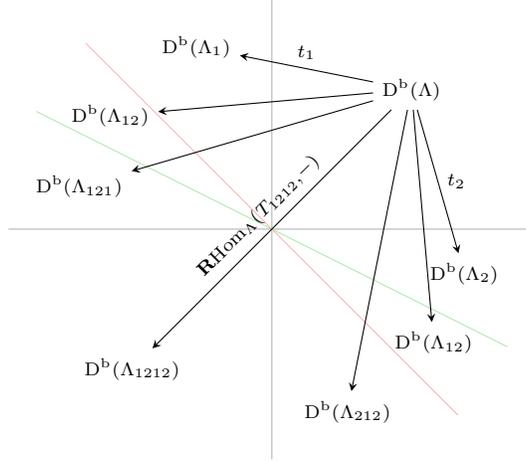
\[
\begin{array}{c}
\begin{tikzpicture}
\node at (0,0) 
{\begin{tikzpicture}[scale=1.75,bend angle=15, looseness=1,>=stealth]
\coordinate (A1) at (135:2cm);
\coordinate (A2) at (-45:2cm);
\coordinate (B1) at (153.435:2cm);
\coordinate (B2) at (-26.565:2cm);
\draw[red!30] (A1) -- (A2);
\draw[green!70!black!30] (B1) -- (B2);
\draw[black!30] (-2,0)--(2,0);
\draw[black!30] (0,-1.75)--(0,1.75);
%%nodes
\node (C+) at (45:1.5cm) {$\scriptstyle \Db(\Lambda)$};
\node (C1) at (112.5:1.5cm) {$\scriptstyle \Db(\Lambda_1)$};
\node (C2) at (145:1.5cm) {$\scriptstyle \Db(\Lambda_{12})$};
\node (C3) at (167.5:1.5cm) {$\scriptstyle \Db(\Lambda_{121})$};
\node (C-) at (225:1.5cm) {$\scriptstyle \Db(\Lambda_{1212})$};
\node (C4) at (-67.5:1.5cm) {$\scriptstyle \Db(\Lambda_{212})$};
\node (C5) at (-35:1.5cm) {$\scriptstyle \Db(\Lambda_{12})$};
\node (C6) at (-13:1.5cm) {$\scriptstyle \Db(\Lambda_2)$};
%%arrows
\draw[->]  (C+) -- node[above] {$\scriptstyle t_1$} (C1);
\draw[->]  (C+) -- (C2);
\draw[->]  (C+) -- (C3);
\draw[->]  (C+)-- node[above,rotate=45] {$\scriptstyle\RHom_\Lambda(T_{1212},-)$} (C-);
\draw[->]  (C+) --  (C4);
\draw[->]  (C+) -- (C5);
\draw[->]  (C+) -- node[right] {$\scriptstyle t_2$} (C6);
%%%labels
%\node at (127:1.675cm) {$\scriptstyle t_2$};
%\node at (157:1.675cm) {$\scriptstyle t_1$};
%\node at (198:1.6cm) {$\scriptstyle t_2$};
\end{tikzpicture}};
\end{tikzpicture}
\end{array}
\]
There are similar additional functors emerging from each of the other chambers.
\end{example}

\subsection{Atoms and the Tilting Order}\label{atoms and order}
Under the flops setup \ref{flops setup}, recall from the last subsection that we associate an algebra $\Lambda=\End_{\mathfrak{R}}(M)$, and a simplicial hyperplane arrangement $\cH_\Lambda$.   The  functor 
\[
\mathbb{F}\colonequals\Hom_{\mathfrak{R}}(M,-)\colon \mod R \to \mod \Lambda
\]
is fully faithful, and  furthermore by \cite[4.17, 5.11]{IW4} induces an injective map
\begin{equation}
\mathbb{F}\colon \{ \mbox{CT }\mathfrak{R}\mbox{-modules}\} 
\to \tilt_0\Lambda\label{Is inj map}
\end{equation}
where recall $\tilt_0\Lambda$ consists of all tilting $\Lambda$-modules containing $P_0=\Hom_{\mathfrak{R}}(M,\mathfrak{R})$ as a summand. By \cite[4.5(1)]{IW6} this map is compatible with mutation.  But since $\mathfrak{R}$ is an isolated singularity, all possible mutations of a fixed CT $\mathfrak{R}$-module $N$ give all possible mutations of $\mathbb{F}N$ in $\tilt_0\Lambda$, hence the finite connected mutation graph of CT $\mathfrak{R}$-modules induces, under $\mathbb{F}$, a finite connected component of $\tilt_0\Lambda$.  By a result of Happel--Unger (adapted and proved in the setting here in \cite[4.9]{IW6}), $\tilt_0\Lambda$ must equal this finite connected component, thus \eqref{Is inj map} is in fact a bijection compatible with mutation.

It follows that the exchange graph $\EG_0\Lambda$  from \S\ref{subsection tilting chambers} equals the exchange graph of CT $\mathfrak{R}$-modules, in a way compatible with mutation.  Hence, by the last subsection, $\EG_0\Lambda$ also equals the $1$-skeleton of $\cH_\Lambda$, and thus the chambers of $\cH_\Lambda$ are indexed by tilting $\Lambda$-modules, in a manner such that two modules that share a codimension one wall are related by a mutation at an indecomposable summand, in the sense of \S\ref{tilting mutation}. We refer the reader to \ref{8 chamb ex 2} for an example.

The following is our main technical lemma, which uses the tilting chambers to establish in the second part that the composition of mutation functors along Deligne normal form is given by a direct tilt.   To avoid confusion, write $D_T$ for the chamber of $\cH_\Lambda$ indexed by $T\in\tilt_0\Lambda$, and write $C_T$ for the chamber \eqref{C T def}. We write $D_{+}\colonequals D_{\Lambda}$.  

\begin{thm}\label{simplicial and functors ok}
Under the setup \ref{flops setup}, for any $S\in\tilt_0\Lambda$,  suppose that $\upalpha\colon D_+\to D_S$ is an atom in $\Gamma_{\cH_\Lambda}$, and choose a decomposition of $\upalpha$ into length one positive paths
\[
\upalpha=\quad D_\Lambda\xrightarrow{s_{i_1}}D_2\to\hdots\to D_{m}\xrightarrow{s_{i_m}}D_{m+1}.
\]
For $i=2,\hdots,m+1$, write $M_i$ for the CT $\mathfrak{R}$-module corresponding to the chamber $D_i$, so that $S= \mathbb{F}M_{m+1}$.
Then the following assertions hold.
\begin{enumerate}
\item\label{mut funct iso 1} 
As tilting $\Lambda$-modules, $\Lambda=\mathbb{F}M>\mathbb{F}M_2>\hdots>\mathbb{F}M_m>\mathbb{F}M_{m+1}=S$.
\item\label{mut funct iso 2}  There is a bimodule isomorphism
\[
\Hom_{\mathfrak{R}}(M_m,M_{m+1})
\stackrel{\bf L}{\otimes}
\hdots
\stackrel{\bf L}{\otimes}
\Hom_{\mathfrak{R}}(M_2,M_3)
\stackrel{\bf L}{\otimes}
\Hom_{\mathfrak{R}}(M,M_2)
\cong
\Hom_{\mathfrak{R}}(M,M_{m+1})
\]
where, reading right to left, the tensors are over $\End_{\mathfrak{R}}(M_i)$ for $i=2,\hdots,m$.
\item\label{mut funct iso 3}  $C_S\colonequals 
\left\{ \sum_{i=1}^n\upvartheta_i [S_i]\, \mid \upvartheta_i> 0\mbox{ for all }1\leq i\leq n
\right\}$ equals $D_S$.
\end{enumerate}
\end{thm}
\begin{proof}
We prove all assertions together.   By induction we can assume that
\begin{equation}
\Lambda=\mathbb{F}M>\hdots>\mathbb{F}M_m, \label{increase induction}
\end{equation}
that $C_X=D_X$ for the tilting modules in \eqref{increase induction}, and that there is a bimodule isomorphism
\begin{equation}
\Hom_{\mathfrak{R}}(M_{m-1},M_{m})
\stackrel{\bf L}{\otimes}
\hdots
\stackrel{\bf L}{\otimes}
\Hom_{\mathfrak{R}}(M,M_2)
\cong
\Hom_{\mathfrak{R}}(M,M_{m})
\label{bimod induction}
\end{equation}
since the case $m=1$ is clear.

Certainly the hyperplanes of $\cH_\Lambda$ cannot pass through any chamber, in particular $C_\Lambda=D_\Lambda$.  Write $T\colonequals\mathbb{F}M_{m}$, and $H'$ for the wall separating $C_T$ and $C_S$.  Since $C_T=D_T$ by induction, extending $H'$ to a hyperplane, $H'$ is one of the hyperplanes of $\cH_\Lambda$.  Hence since $[\Lambda]\in C_\Lambda$, and the hyperplanes of $\cH_\Lambda$ cannot pass through $C_\Lambda$, necessarily $[\Lambda]\notin H'$.   

We next crash through the wall $H$ from $D_T$ into $D_S$.  If $D_+$ is not on the same side of $H$ as $D_T$, then $\upalpha$ would have to cross $H$ twice, and so by \ref{2.14} applied to $\cH_\Lambda$,  the path $\upalpha$ would not be an atom.  Hence $D_+$ must be on the same side of $H$ as $D_T$.  Since $[\Lambda]\in C_+=D_+$ and $C_T=D_T$, we conclude that $[\Lambda]$ is on the same side of $H'$ as $C_T$.

 By \ref{DIJ lemmaB} necessarily  $T>S$, i.e.\ $\mathbb{F}M_m>\mathbb{F}M_{m+1}$, so combining with \eqref{increase induction} proves \eqref{mut funct iso 1}.  Next, the induction \eqref{bimod induction} gives a bimodule isomorphism
\[
\Hom_{\mathfrak{R}}(M_{m},M_{m+1})
\stackrel{\bf L}{\otimes}
(\hdots
\stackrel{\bf L}{\otimes}
\Hom_{\mathfrak{R}}(M,M_2))
\cong
\Hom_{\mathfrak{R}}(M_{m},M_{m+1})
\stackrel{\bf L}{\otimes}
\Hom_{\mathfrak{R}}(M,M_{m}),
\]
so to prove \eqref{mut funct iso 2} it suffices to show that there is a bimodule isomorphism
\begin{equation}
\Hom_{\mathfrak{R}}(M_{m},M_{m+1})
\stackrel{\bf L}{\otimes}
\Hom_{\mathfrak{R}}(M,M_{m})
\cong
\Hom_{\mathfrak{R}}(M,M_{m+1}).
\label{to show 2}
\end{equation}
Applying \ref{flow big to small} with $T=\mathbb{F}M_m$, $\Gamma=\End_{\mathfrak{R}}(M_{m})$ and $\upnu_i\Gamma=\Hom_{\mathfrak{R}}(M_m,M_{m+1})$ shows that the left hand side of \eqref{to show 2} is concentrated in degree zero, so to prove \eqref{mut funct iso 2} it suffices to show that there is a bimodule isomorphism
\begin{equation}
\Hom_{\mathfrak{R}}(M_{m},M_{m+1})
\otimes
\Hom_{\mathfrak{R}}(M,M_{m})
\cong
\Hom_{\mathfrak{R}}(M,M_{m+1}).
\label{to show 3}
\end{equation}
But there is a chain of isomorphisms
\[
\Hom_{\mathfrak{R}}(M_{m},M_{m+1})
\otimes
\Hom_{\mathfrak{R}}(M,M_{m})
\xrightarrow{\sim}
\Hom_{\Lambda}(T,\mathbb{F}M_{m+1})\otimes T
\xrightarrow{\sim}
\mathbb{F}M_{m+1}
\]
where the first is reflexive equivalence $g\otimes f\mapsto (g\circ-)\otimes f$, and the second is the adjunction from the derived equivalence (using the last statement in \ref{flow big to small}), which takes $\upvarphi\otimes t\mapsto \upvarphi(t)$.  Composing the above shows that there is an isomorphism \eqref{to show 3}, given by $g\otimes f\mapsto g\circ f$.  By inspection this an isomorphism in the category of bimodules, proving \eqref{mut funct iso 2}.

Finally, to prove \eqref{mut funct iso 3}, note that the bimodule isomorphism in \eqref{mut funct iso 2} induces a functorial isomorphism between $\RHom_\Lambda(\mathbb{F}M_{m+1},-)=\RHom_\Lambda(S,-)$ and the composition
\begin{equation} 
\Db(\mod\Lambda)\xrightarrow{t_{i_1}}\hdots\xrightarrow{t_{i_m}}\Db(\mod\Lambda_{m+1}).\label{simple mut chain}
\end{equation}
Writing $\Lambda_{m+1}=P_0\oplus Q_1\oplus\hdots \oplus Q_n$, it is easy to see that tracking 
\begin{equation}
\{ \sum_i\upvartheta_i[Q_i] \mid \upvartheta_i>0, i=1,\hdots, n\}\label{C Gamma}
\end{equation}
 through the inverse of $\RHom_\Lambda(S,-)$ gives 
\[
\{ \sum_i\upvartheta_i[S_i] \mid \upvartheta_i>0, i=1,\hdots, n\}=C_S. 
\]
By the functorial isomorphism, this must give the same answer as tracking \eqref{C Gamma} through the inverse of \eqref{simple mut chain}.  We thus claim that tracking \eqref{C Gamma} through the inverse of \eqref{simple mut chain} gives $D_S$, as then $D_S=C_S$ and the result follows.

On one hand, by the definition of the mutation functors, tracking \eqref{C Gamma} through the inverse of \eqref{simple mut chain} precisely follows the moduli-tracking rules laid out in \cite[5.14, 5.15]{HomMMP}.   On the other hand, it is known \cite[5.25]{HomMMP} that after possibly replacing some of the
\[
t_j=\RHom_{\End_{\mathfrak{R}}(N)}(\Hom_{\mathfrak{R}}(N,\upnu_jN),-)
\] 
in \eqref{simple mut chain} by 
\[
t_j'\colonequals -\otimes^{\bf L}_{\End_{\mathfrak{R}}(N)}\Hom_{\mathfrak{R}}(\upnu_jN,N),
\] 
tracking \eqref{C Gamma} back through the inverse of the replacement chain does indeed give the simplicial cone $D_S$.  Crucially, since the combinatorial rules for tracking through $t_j$ and through $t_j'$ are the same in this flops setting (see \cite[5.15]{HomMMP}), the replacements do not matter, and so tracking \eqref{C Gamma} through the inverse of \eqref{simple mut chain}  also gives $D_S$, as required.
\end{proof}
%
%\begin{cor}
%\begin{enumerate}
%\item In particular, the tilting module chambers $\{ C_T\mid T\in\tilt_0\Lambda\}$ are the chambers of the simplicial hyperplane arrangement $\cH_\Lambda$.
%\item There is a functorial isomorphism
%\begin{eqnarray}
%t_{{i_m}}\circ\hdots\circ t_{{i_1}}\cong  \RHom_{\Lambda}(\upnu_{\!\upalpha}\Lambda,-)\equalscolon t_\upalpha\label{funct iso torsion}
%\end{eqnarray}
%with $\upnu_{\!\upalpha}\Lambda\in\tilt\Lambda$. 
%\end{enumerate}
%\end{cor}
%

\begin{remark}\label{any atom will do}
We remark that the initial choice of decomposition of $\upalpha$ in \ref{simplicial and functors ok} does not matter, as the theorem shows that all choices are functorially isomorphic to $t_\upalpha$.  
\end{remark}

\begin{example}\label{8 chamb ex 2}
Continuing the flopping contraction example in \ref{wall crossing functors}, the chambers of $\cH_\Lambda$ can be indexed by elements of $\tilt_0\Lambda$, as illustrated in the left hand side of the following picture, where $\upnu_{i_2i_1}\Lambda=\Hom_{\mathfrak{R}}(M,\upnu_{i_2i_1}M)$ etc.  The ordering, which is illustrated in the right hand side, is forced by \ref{simplicial and functors ok}\eqref{mut funct iso 1}.
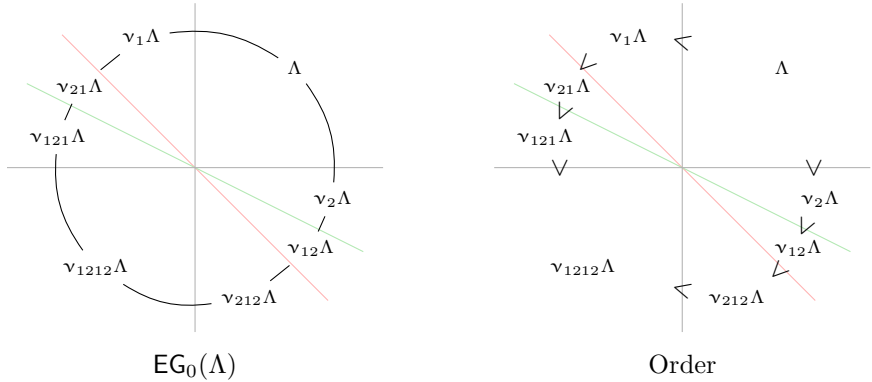
\[
\begin{array}{ccc}
\begin{array}{c}
\begin{tikzpicture}
\node at (5,0) {$\begin{tikzpicture}[scale=1.25,>=stealth,bend angle=20]
\coordinate (A1) at (135:2cm);
\coordinate (A2) at (-45:2cm);
\coordinate (B1) at (153.435:2cm);
\coordinate (B2) at (-26.565:2cm);
\coordinate (C1) at (161.565:2cm);
\coordinate (C2) at (-18.435:2cm);
\draw[red!30] (A1) -- (A2);
\draw[green!70!black!30] (B1) -- (B2);
%\draw[blue] (C1) -- (C2);
\draw[black!30] (-2,0)--(2,0);
\draw[black!30] (0,-1.75)--(0,1.75);
%%nodes
\node (C+) at (45:1.5cm) {$\scriptstyle \Lambda$};
\node (C1) at (112.5:1.5cm) {$\scriptstyle \upnu_1\Lambda$};
\node (C2) at (145:1.5cm) {$\scriptstyle \upnu_{21}\Lambda$};
\node (C3) at (167.5:1.5cm) {$\scriptstyle \upnu_{121}\Lambda$};
\node (C-) at (225:1.5cm) {$\scriptstyle \upnu_{1212}\Lambda$};
\node (C4) at (-67.5:1.5cm) {$\scriptstyle \upnu_{212}\Lambda$};
\node (C5) at (-35:1.5cm) {$\scriptstyle \upnu_{12}\Lambda$};
\node (C6) at (-13:1.5cm) {$\scriptstyle \upnu_2\Lambda$};
%%edges
\draw[bend right]  (C+) to (C1);
\draw  (C1) to (C2);
\draw  (C3) to (C2);
\draw[bend left]  (C-) to  (C3);
\draw[bend left]  (C+) to  (C6);
\draw (C6) to  (C5);
\draw (C5) to  (C4);
\draw[bend left]  (C4) to  (C-);
\end{tikzpicture}$};
\end{tikzpicture}\\
\EG_0(\Lambda)
\end{array}
&&
\end{array}
\begin{array}{c}
\begin{tikzpicture}
\node at (5,0) {$\begin{tikzpicture}[scale=1.25,>=stealth]
\coordinate (A1) at (135:2cm);
\coordinate (A2) at (-45:2cm);
\coordinate (B1) at (153.435:2cm);
\coordinate (B2) at (-26.565:2cm);
\coordinate (C1) at (161.565:2cm);
\coordinate (C2) at (-18.435:2cm);
\draw[red!30] (A1) -- (A2);
\draw[green!70!black!30] (B1) -- (B2);
%\draw[blue] (C1) -- (C2);
\draw[black!30] (-2,0)--(2,0);
\draw[black!30] (0,-1.75)--(0,1.75);
%%nodes
\node (C+) at (45:1.5cm) {$\scriptstyle \Lambda$};
\node (C1) at (112.5:1.5cm) {$\scriptstyle \upnu_1\Lambda$};
\node (C2) at (145:1.5cm) {$\scriptstyle \upnu_{21}\Lambda$};
\node (C3) at (167.5:1.5cm) {$\scriptstyle \upnu_{121}\Lambda$};
\node (C-) at (225:1.5cm) {$\scriptstyle \upnu_{1212}\Lambda$};
\node (C4) at (-67.5:1.5cm) {$\scriptstyle \upnu_{212}\Lambda$};
\node (C5) at (-35:1.5cm) {$\scriptstyle \upnu_{12}\Lambda$};
\node (C6) at (-13:1.5cm) {$\scriptstyle \upnu_2\Lambda$};
%%order
\node at (90:1.35cm) [rotate=-15] {$<$};
\node at (133:1.5cm) [rotate=45] {$<$};
\node at (155:1.4cm) [rotate=65] {$<$};
\node at (180:1.3cm) [rotate=90] {$<$};
\node at (0:1.4cm) [rotate=90] {$<$};
\node at (-47:1.5cm) [rotate=45] {$<$};
\node at (-26:1.45cm) [rotate=65] {$<$};
\node at (-90:1.3cm) [rotate=-15] {$<$};
\end{tikzpicture}$};
\end{tikzpicture}\\
\mbox{Order}
\end{array}
\]
Since the positive path below corresponding to the composition $t_2t_1t_2$ is an atom, it also follows from \ref{simplicial and functors ok}\eqref{mut funct iso 2} (applied to $\Lambda_1$) that the composition $t_2t_1t_2$ is functorially isomorphic to the direct functor shown
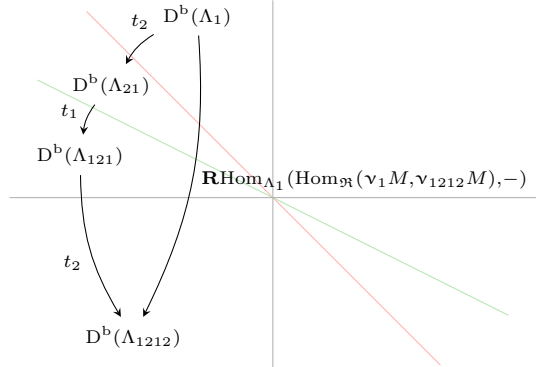
\[
\begin{array}{c}
\begin{tikzpicture}
\node at (0,0) 
{\begin{tikzpicture}[scale=1.75,bend angle=15, looseness=1,>=stealth]
\coordinate (A1) at (135:2cm);
\coordinate (A2) at (-45:1.8cm);
\coordinate (B1) at (153.435:2cm);
\coordinate (B2) at (-26.565:2cm);
\draw[red!30] (A1) -- (A2);
\draw[green!70!black!30] (B1) -- (B2);
\draw[black!30] (-2,0)--(2,0);
\draw[black!30] (0,-1.3)--(0,1.5);
%%nodes
\node (C1) at (112.5:1.5cm) {$\scriptstyle\Db(\Lambda_1)$};
\node (C2) at (145:1.5cm) {$\scriptstyle\Db(\Lambda_{21})$};
\node (C3) at (167.5:1.5cm) {$\scriptstyle\Db(\Lambda_{121} )$};
\node (C-) at (225:1.5cm) {$\scriptstyle\Db(\Lambda_{1212} )$};
%%arrows
\draw[->, bend right]  (C1) to (C2);
\draw[->, bend right]  (C2) to (C3);
\draw[->, bend right]  (C3) to (C-);
\draw[->, bend left]  (C1) to node [right] {$\scriptstyle \RHom_{\Lambda_1}(\Hom_{\mathfrak{R}}(\upnu_1M,\upnu_{1212}M),-)$} (C-);
%%%labels
\node at (127:1.675cm) {$\scriptstyle t_2$};
\node at (157:1.675cm) {$\scriptstyle t_1$};
\node at (198:1.6cm) {$\scriptstyle t_2$};
\end{tikzpicture}};
\end{tikzpicture}
\end{array}
\]
\end{example}

%The following two remarks explain the geometric significance of the $t_\upalpha$.
%
%\begin{remark}
%Suppose that $t=2$ and $\Lambda$ is the preprojective algebra of an extended Dynkin diagram.  It is well-known that $\Lambda$ is derived equivalent to the minimal resolution of the associated Kleinian singularity.  If $\ell(\upalpha)=1$, then $t_\upalpha$ defined in \ref{t for length one} is by definition the mutation functor, which is known to be functorially isomorphic to a spherical twist \cite{DW1}.  
%This allows us to recover Brav--Thomas \cite{BT} in Appendix~\ref{BT appendix}.  More generally, later we will also consider algebras of the form $e\Lambda e$, which are derived equivalent to certain partial crepant resolutions. In this case there is often no `twist' description of $t_\upalpha$, but we will still obtain a faithful action. 
%\end{remark}
%
%
%
%
%
%
%Suppose that $d=3$ and $R$ is cDV with maximal modifying algebra $\Lambda$.  Then $\Lambda$ is necessarily derived equivalent to some minimal model of $\Spec R$.  By \cite[4.23]{HomMMP}, if $\ell(\upalpha)=1$, then $t_\upalpha$ is either functorially isomorphic to the flop functor when $\dim_\mathbb{C}\Lambda_{\con}<\infty$, or else is given by some noncommutative twist.  However, when $\ell(\upalpha)>1$, i.e.\ $\upalpha$ is not a simple wall crossing, then $t_\upalpha$ has no good `twist' or flop description.  This is what makes iteration difficult, and why \ref{mut funct iso} becomes important.

For future use, a useful corollary of \ref{simplicial and functors ok} is the following.
\begin{cor}\label{cohom}
Under the assumptions \ref{flops setup}, let $\upalpha\colon C\to D$ be an atom, and let $N\in\mod\Lambda_C$. Then $H^i(t_{\upalpha}N)=0$ for all $i\neq 0,1$.
\end{cor}

\section{Tracking via Torsion Pairs}\label{tracking torsion pairs}

Under the flops setup of \ref{flops setup}, suppose that $C$ is a chamber of $\cH_\Lambda$.  It follows from \ref{simplicial and functors ok} applied to $\Lambda_C=\End_{\mathfrak{R}}(N_C)$ that if $\upalpha\colon C\to D$ is an atom, then the composition of mutation functors along the path $\upalpha$ is functorially isomorphic to $t_\upalpha= \RHom_{\Lambda_C}(T_\upalpha,-)$ where $T_\upalpha\colonequals\upnu_{\!\upalpha}\Lambda_C\in\tilt_0\Lambda_C$.  We will use this implicitly from now on.

As is standard (see e.g.\ \cite[2.7]{SY}),  $T_\upalpha$ induces two torsion pairs, which restrict to torsion pairs on finite length modules $\fl\Lambda_C$ and $\fl\Lambda_D$.  These are $(\cT_\upalpha,\cF_\upalpha)$ and $(\cX_\upalpha,\cY_\upalpha)$, where
\begin{align*}
\cT_\upalpha&\colonequals \{ N\in\fl \Lambda_C\mid \Ext^1_{\Lambda_C}(T_\upalpha,N)=0\}\\
\cF_\upalpha&\colonequals \{ N\in\fl \Lambda_C\mid \Hom_{\Lambda_C}(T_\upalpha,N)=0\},
\end{align*} 
and 
\begin{align*}
\cX_\upalpha&\colonequals \{ N\in\fl \Lambda_D\mid N\otimes_{\Lambda_D} T_\upalpha=0\}\\
\cY_\upalpha&\colonequals \{ N\in\fl \Lambda_D\mid \Tor_1^{\Lambda_D}(N,T_\upalpha)=0\}.
\end{align*} 
The Brenner--Butler Theorem for finite dimensional algebras (proved in the module-finite setting here in \cite[2.9]{SY}) asserts that these tilting modules not only induce the above two torsion pairs, but also induce the following categorical equivalences:
\begin{eqnarray}
\begin{array}{ccc}
\begin{array}{c}
\begin{tikzpicture}
\node (A) at (-0.75,0) {$\cT_\upalpha$};
\node (B) at (2.15,0) {$\cY_\upalpha$};
\draw[->] (-0.3,0.1) -- node [above] {$\scriptstyle \Hom_{\Lambda_C}(T_\upalpha,-)$} (1.7,0.1);
\draw [<-] (-0.3,-0.1) -- node [below] {$\scriptstyle \otimes_{\Lambda_D} T_\upalpha$} (1.7,-0.1);
\end{tikzpicture}
\end{array}
&
\quad\mbox{and}\quad\quad
& 
\begin{array}{c}
\begin{tikzpicture}
\node (A) at (-0.75,0) {$\cF_\upalpha$};
\node (B) at (2.15,0) {$\cX_\upalpha.$};
\draw[->] (-0.3,0.1) -- node [above] {$\scriptstyle \Ext^1_{\Lambda_C}(T_\upalpha,-)$} (1.7,0.1);
\draw [<-] (-0.3,-0.1) -- node [below] {$\scriptstyle \Tor_1^{\Lambda_D}(-,T_\upalpha)$} (1.7,-0.1);
\end{tikzpicture}
\end{array}
\end{array}
\label{TF and XY}
\end{eqnarray}

To control the functors $t_\upalpha$ requires us to track various objects, which we do here.  The following lemma is a standard fact about Deligne normal form, which precisely mirrors the Coxeter version.

\begin{lemma}\label{standard Cox}
Suppose that $\upalpha$ is an atom.  Then
\begin{enumerate}
\item\label{standard Cox 1} $\upalpha s_i$ is an atom $\iff$ $s_i\notin\Begin(\upalpha)$.
\item\label{standard Cox 2} $s_j \upalpha$ is an atom $\iff$ $s_j\notin\Ends(\upalpha)$.
\end{enumerate}
\end{lemma}
\begin{proof}
For lack of a suitable reference, we give the proof of \eqref{standard Cox 1}, with \eqref{standard Cox 2} being similar.\\
($\Rightarrow$) is clear, using \ref{2.14}.\\
($\Leftarrow$) We prove the contrapositive.  Suppose that the composition
\[
A\xrightarrow{s_i} B \xrightarrow{\upalpha}C
\] 
is not an atom, and write $H$ for the hyperplane separating $A$ and $B$.  By \ref{2.14}, $\upalpha s_i$ must cross some hyperplane at least twice.  But since $\upalpha$ is an atom, again by \ref{2.14}, the hyperplanes that $\upalpha$ crosses must be distinct.  Hence the only possibility is that $\upalpha s_i$ crosses $H$ precisely twice.

In particular, $\upalpha$ must cross $H$, and so since it cannot cross $H$ twice by \ref{2.14}, $t(\upalpha)=C$ must be on the same side of $H$ as $A$.  If we write $\upbeta$ for the smallest positive path (atom) from $A$ to $C$, then $\upbeta$ cannot cross $H$ by \ref{2.14}, since $A$ and $C$ lie on the same side of $H$.  Since $s_i\colon B\to A$ obviously only crosses $H$, it follows again by \ref{2.14} that the composition
\[
B\xrightarrow{s_i}A\xrightarrow{\upbeta}C
\] 
is an atom. Hence $\upalpha\sim\upbeta s_i$, since both are atoms from $B$ to $C$, and so $s_i\in\Begin(\upalpha)$.
\end{proof}

\begin{notation}\label{not new}
In each chamber $D$ of $\EG_0(\Lambda)$ there is an algebra $\Lambda_D$ with precisely $n+1$-simples.  By abuse of notation we will denote these simples $S_0,S_1,\hdots, S_n$, where $S_0$ always corresponds to $P_0$, and performing the simple wall crossing $s_i$ corresponds to the tilting mutation at the projective cover of $S_i$.  We will use the same notation $S_i$ for every $\Lambda_D$, and will often consider $\cS\colonequals \bigoplus_{i=0}^nS_i$, with it being implicit from the context which $\Lambda_D$ to view this as a module over. 
\end{notation}

\begin{lemma}\label{Si under inverse}
Under the assumptions \ref{flops setup}, $t_i(S_i)\cong S_i[-1]$ for all $1\leq i\leq n$.
\end{lemma}
\begin{proof}
Say $s_i\colon C\to D$, so that $t_i=\RHom_{\Lambda_C}(\upnu_i\Lambda_C,-)$.  Since $\Lambda_C>\upnu_i\Lambda_C$, as in Appendix~\ref{tilting appendix} there exists a short exact sequence
\[
0\to P_i\to P'\to C_i\to 0
\]
with $P'\in\add\frac{\Lambda_C}{P_i}$ such that $\upnu_i\Lambda_C=\frac{\Lambda_C}{P_i}\oplus C_i$.  Applying $\Hom_{\Lambda_C}(-,S_i)$ to the above sequence yields the result, exactly as in \cite[4.15(2)]{HomMMP}.
\end{proof}

For our purposes later, we require more than \ref{Si under inverse}, namely for atoms $\upalpha\colon C\to D$ we need to track all summands of $\cS$ under the inverse functor $t_\upalpha^{-1}\cong -\otimes^{\bf L}_{\Lambda_D}T_\upalpha$.  Since $(\cX_\upalpha,\cY_\upalpha)$ is a torsion pair on $\fl\Lambda_D$, and each $S_i$ is simple, either $S_i\in\cX_\upalpha$ or $S_i\in\cY_\upalpha$.  Using the categorical equivalences \eqref{TF and XY} it thus follows that
\begin{eqnarray}
t_\upalpha^{-1}(S_i)=
\left\{
\begin{array}{cl}
\Tor_1^{\Lambda_D}(S_i,T_\upalpha)[1]&\mbox{if }S_i\in\cX_\upalpha\\[1mm]
S_i\otimes_{\Lambda_D} T_\upalpha&\mbox{if }S_i\in\cY_\upalpha.
\end{array}
\right.\label{inverse Si}
\end{eqnarray}
In the top case, $t_\upalpha^{-1}(S_i)$ is the shift of a module in $\cF_{\upalpha}$, and in the bottom case $t_\upalpha^{-1}(S_i)$ is a module in $\cT_{\upalpha}$.

The following is our key preparatory lemma, which says that the torsion pairs $(\cT_\upalpha,\cF_\upalpha)$ and $(\cX_\upalpha,\cY_\upalpha)$ detect both how $\upalpha$ starts, and how $\upalpha$ ends.

\begin{lemma}\label{Key SY lemma}
Under the assumptions \ref{flops setup}, suppose that $\upalpha\colon C\to D$ is an atom. Then for $i\neq 0$, the following statements hold.
\begin{enumerate}
\item\label{Key SY lemma 1}
$S_i\in\cF_{\upalpha}\iff s_i\in\Begin(\upalpha)$.
\item\label{Key SY lemma 2}  
$S_i\in\cT_{\upalpha}\iff s_i\notin\Begin(\upalpha)$.
\item\label{Key SY lemma 3}  
$S_i\in\cX_{\upalpha}\iff s_i\in\Ends(\upalpha)$.
\item\label{Key SY lemma 4}  
$S_i\in\cY_{\upalpha}\iff s_i\notin\Ends(\upalpha)$.
\end{enumerate}
\end{lemma}
\begin{proof}
We prove \eqref{Key SY lemma 1}, with all others being similar. \\
($\Leftarrow$) Suppose that $\upalpha$ starts with $s_i$, and write $\upalpha=\upbeta\circ s_i$.  Then $t_\upalpha(S_i)\stackrel{\ref{Si under inverse}}{=}t_\upbeta(S_i)[-1]$. Applying \ref{cohom} to both sides, it follows that $H^j(t_\upalpha(S_i))=0$ for all $j\neq 1$, so $S_i\in \cF_\upalpha$. \\
($\Rightarrow$) Suppose that $\upalpha$ does not start with $s_i$, then by \ref{standard Cox} $\upalpha\circ s_i$ is still a reduced expression of Deligne normal form.  Hence $t_{\upalpha s_i}=t_\upalpha\circ t_{s_i}$, and so  $t_{\upalpha s_i}(S_i)\stackrel{\ref{Si under inverse}}{=}t_{\upalpha}(S_i)[-1]$.  Thus  $t_{\upalpha}(S_i)=t_{\upalpha s_i}(S_i)[1]$, so again applying \ref{cohom} to both sides, it follows that $H^j(t_\upalpha(S_i))=0$ for all $j\neq 0$, so $S_i\in \cT_\upalpha$.  In particular, $S_i\notin \cF_\upalpha$.
\end{proof}

\begin{lemma}\label{S0 ok}
Under the assumptions \ref{flops setup}, suppose that $\upalpha\colon C\to D$ is an atom.  Then $S_0\in\cT_\upalpha$ and $S_0\in\cY_\upalpha$.
\end{lemma}
\begin{proof}
The first statement holds since $P_0$ is a summand of $T_\upalpha$, so $\Hom_{\Lambda_C}(T_\upalpha,S_0)\neq 0$.  Thus  $S_0\notin\cF_{\upalpha}$ and so since $S_0$ is simple, necessarily $S_0\in\cT_{\upalpha}$.

For the second statement is similar, but uses the duality on tilting modules, so we sketch the proof. To ease notation set $A\colonequals \Lambda_C$, $B\colonequals \Lambda_D$, and $T\colonequals T_\upalpha$.  By convention the simple right $A$-module $S_0$ corresponds to the indecomposable projective $P_0$ of $A$, so consider the idempotent $e_0$ such that $P_0=e_0A$.   Similarly, $B\cong\End_A(T)$ has an idempotent $e_0'$ corresponding to the summand $e_0A$ in the decomposition $T=e_0A\oplus X$ as right $A$-modules.  By convention $S_0$ is the simple right $B$-module corresponding to $e_0'B$, so that the $k$-dual $DS_0$ is the simple left $B$-module corresponding to $Be_0'$.  It follows that $\Hom_{B^{\op}}(Be_0',DS_0)\neq 0$.

We first claim that $Be_0'$ is a summand of ${}_BT$.   By construction, it is clear that $Be_0'=\Hom_A(e_0A,T)$ as left $B$-modules.  As in  \cite[p33]{BBtilting}, the functor
 \[
 {}^\star(-)\colonequals \Hom_A(-,T)\colon \mod A\to \mod B^{\op}
 \]
clearly takes $A_A\mapsto {}_BT$, and thus since $e_0A$ is a summand of $A_A$, by applying ${}^\star(-)$ we see that ${}^\star(e_0A)\cong Be_0'$ is a summand of ${}_BT$.

Now by \cite[VI.5.1]{CE} there is an isomorphism
\[
D(S_0\otimes_B Be_0')\cong \Hom_{B^{\op}}(Be_0',DS_0),
\]
which is non-zero by above.  Thus $S_0\otimes_B Be_0'\neq 0$.  Since by above $S_0\otimes_B T$ has summand $S_0\otimes_B Be_0'$, it follows that $S_0\otimes_BT\neq 0$.  Hence $S_0\notin\cX_\upalpha$, so again since $S_0$ is simple, necessarily $S_0\in\cY_\upalpha$. 
\end{proof}

\begin{cor}\label{final cor}
Under the assumptions \ref{flops setup}, suppose that $\upalpha\colon C\to D$ is an atom.  If $N\in\cF_\upalpha$ is nonzero, then there exists some $j\neq 0$ such that $\upalpha$ starts with $s_j$, and further $\Hom_{\Lambda_C}(S_j,N)\neq 0$. 
\end{cor}
\begin{proof}
Certainly $N$ is filtered by simples, so there exists some $0\leq j\leq n$ with $S_j\hookrightarrow N$.  In particular $\Hom(S_j,N)\neq 0$. Since $\cF_{\upalpha}$ is closed under submodules $S_j\in\cF_\upalpha$, and so by \ref{S0 ok} necessarily $j\neq 0$. The result then follows from \ref{Key SY lemma}\eqref{Key SY lemma 1}.
\end{proof}

\section{Proof of Faithfulness}\label{faithful proof section}

Keeping the notation in the previous sections, under the flops setup of \ref{flops setup}, recall from \ref{not new} that every chamber $D$ has an associated algebra $\Lambda_D$ and simple modules $S_0, S_1,\hdots,S_n$, and we set $\cS\colonequals \bigoplus_{i=0}^nS_i$.  As in the Conventions, we write $[a,b\kern 1pt]_t=\Hom_{\Db(\Lambda_D)}(a,b[t])$.   Although the $D$ is suppressed in this notation, it will be clear from the context in which category to view $\cS$.

We will reduce to a key technical lemma in \ref{BT main}, which is an analogue of \cite[Prop.\ 3.1]{BT}.  The key point in Brav--Thomas is to first find an object  $b$ such that  
\begin{equation}
[\cS, b\kern 1pt]_{\geq d+1}= 0,\label{to control}
\end{equation}
%Unfortunately the existence of such a $j$ is specific to the setting of spherical twists, so below we work in the following axiomatic setting. 
%\begin{assum}\label{faithful assumptions}
%Let $\Lambda$ be a basic $\mathfrak{R}$-algebra, where $\mathfrak{R}$ is a complete local domain.  We assume the following.
%\begin{enumerate}
%\item\label{faithful assumptions 2}The assumptions in \ref{main assumptions} hold.  
%\item\label{faithful assumptions 1} There is $d\in\mathbb{N}$ such that for each chamber $D$, there exists $b\in\Db(\mod\Lambda_D)$ with
%\begin{eqnarray}
%[S_i,b\kern 1pt]_d\neq0\mbox{ for all }0\leq i\leq n,\quad\mbox{and}\quad[\cS,b\kern 1pt]_{\geq d+1}=0.\label{start induction}
%\end{eqnarray}
%\end{enumerate}
%\end{assum}
%We will prove the geometric settings of interest admit such an object $b$, and also satisfy the second condition, later in \ref{assum ok in geo settings}.
where $d=\dim \mathfrak{R}$. For this  there are many choices. To ensure that the method below can be used in future papers to cover situations where $\Lambda$ has infinite global dimension (or flopping contractions $U\to\Spec \mathfrak{R}$ where $U$ need not be smooth), throughout we choose $b=\Lambda$, as is justified in the following lemma.

\begin{lemma}\label{b ok for modifying}
Suppose that $\mathfrak{R}$ is a $d$-dimensional complete local Gorenstein ring, and that $\Lambda\cong\End_{\mathfrak{R}}(M)$ for some $M\in\refl \mathfrak{R}$, with $\Lambda\in\CM \mathfrak{R}$ (that is, $\Lambda$ is a modifying $\mathfrak{R}$-algebra).  Then $b\colonequals \Lambda$ satisfies
\begin{eqnarray}
[S_i,b\kern 1pt]_d\neq0\mbox{ for all }0\leq i\leq n,\quad\mbox{and}\quad[\cS,b\kern 1pt]_{\geq d+1}=0.\label{start induction 22}
\end{eqnarray} 
\end{lemma}
\begin{proof}
We know that $\Ext^t_\Lambda(S_i,b)\colonequals\Ext^t_\Lambda(S_i,\Lambda)\cong\Ext^t_\mathfrak{\mathfrak{R}}(S_i,\mathfrak{R})$, where the last isomorphism is \cite[3.4(5)]{IR}.  Hence by local duality
\[
\depth_\mathfrak{R}S_i=d-\sup\{t\geq 0\mid [S_i,b\kern 1pt]_t\neq 0\}.
\]
Clearly, being finite length, $\depth_\mathfrak{R}S_i=0$, so we deduce that \eqref{start induction 22} holds.
\end{proof}

%\begin{assum}\label{faithful assumptions}
%Let $\Lambda$ be a basic $\mathfrak{R}$-algebra, where $\mathfrak{R}$ is a complete local domain.  We assume the following.
%\begin{enumerate}
%\item\label{faithful assumptions 2}The assumptions in \ref{main assumptions} hold.  
%\item\label{faithful assumptions 1} There is $d\in\mathbb{N}$ such that for each chamber $D$, there exists $b\in\Db(\mod\Lambda_D)$ with
%\begin{eqnarray}
%[S_i,b\kern 1pt]_d\neq0\mbox{ for all }0\leq i\leq n,\quad\mbox{and}\quad[\cS,b\kern 1pt]_{\geq d+1}=0.\label{start induction}
%\end{eqnarray}
%\end{enumerate}
%\end{assum}

%We will prove the geometric settings of interest admit such an object $b$, and also satisfy the second condition, later in \ref{assum ok in geo settings}.

\subsection{The Main Result}\label{section main}
Throughout this subsection we will work under the setting of \ref{flops setup}, and write $b\colonequals \Lambda$.  The initial step requires the following elementary lemma.

\begin{lemma}\label{trivial}
Suppose that $0\neq N\in\fl\Lambda_D$.
\begin{enumerate}
\item\label{trivial 1} If $y\in\Db(\mod\Lambda_D)$ is such that $[\cS,y\kern 1pt]_{\geq p}=0$, then $[N,y\kern 1pt]_{\geq p}=0$. 
\item\label{trivial 2}  $[N,b\kern 1pt]_{d}\neq 0$ and $[N,b\kern 1pt]_{\geq d+1}=0$.
\end{enumerate}
\end{lemma}
\begin{proof}
(1) is an easy induction on the length of the filtration of $N$, using the long exact sequence from $[-,y\kern 1pt]$.\\
(2) By \ref{b ok for modifying} $[\cS,b\kern 1pt]_{\geq d+1}=0$, so the second statement is a consequence of \eqref{trivial 1}. The first also  follows by an induction on the length of the filtration of $N$, using $[S_i,b\kern 1pt]_d\neq 0$ and $[S_i,b \kern 1pt]_{d+1}=0$ for all $0\leq i\leq n$.
\end{proof}

Now for $\upalpha\in\mathds{G}_{\cH}^+$, we can decompose $\upalpha$ into length one atoms $
\upalpha=s_{i_n}\hdots s_{i_1}$
and define $t_\upalpha\colonequals t_{i_n}\circ\hdots \circ t_{i_1}$ (where the $t_{i_t}$ are defined in \ref{t for length one}), or alternatively we can decompose $\upalpha$ into Deligne normal form $\upalpha=\upalpha_k\hdots \upalpha_1$ and define $t_{\upalpha}\colonequals t_{\upalpha_k}\circ\hdots\circ t_{\upalpha_1}$  (where the $t_{\upalpha_i}$ are also defined in \ref{t for length one}).  The crucial point in the proof of faithfulness is that by \ref{simplicial and functors ok}\eqref{mut funct iso 2}  these yield the same functor.

The following is our analogue of the main technical lemma of Brav--Thomas \cite[Prop.\ 3.1]{BT}.  Using torsion pairs, the proof only needs to induct on the number of Deligne factors, whereas Brav--Thomas use a more complicated double induction.

\begin{prop}\label{BT main}
Let $1\neq \upalpha\in \mathds{G}^+_{\cH}$ have Deligne normal form $\upalpha=\upalpha_k\circ\hdots\circ \upalpha_1$.  Then
\begin{enumerate}
\item\label{BT main 1} $[\cS,t_\upalpha b\kern 1pt]_{\geq k+d+1}=0$.
\item\label{BT main 2} $[S_i,t_\upalpha b\kern 1pt]_{k+d}\neq 0$ if and only if $i\neq 0$ and the atom $\upalpha_k$ ends (up to the relations in $\mathds{G}^+_{\cH}$) by passing through wall $i$.  In particular $[\cS,t_\upalpha b\kern 1pt]_{k+d}\neq 0$.
\item\label{BT main 3} The maximal $p$ such that $[\cS,t_\upalpha b\kern 1pt]_p\neq 0$ is precisely $p=k+d$.
\end{enumerate}
\end{prop}
\begin{proof}
Statement \eqref{BT main 3} follows immediately from \eqref{BT main 1} and \eqref{BT main 2}, so we prove both \eqref{BT main 1} and \eqref{BT main 2} together using induction on the number of Deligne factors.

\medskip
\noindent
\textbf{Base Case: $k=1$}, i.e.\ $\upalpha$ is an atom.  Since $S_i$ is simple, there are only two cases, namely  $S_i\in\cY_\upalpha$ or $S_i\in\cX_\upalpha$, and using \ref{Key SY lemma} and \ref{S0 ok} we can characterise these:

\medskip
\noindent
(a) $S_i\in\cY_\upalpha$ (equivalently, $i=0$, or $i\neq 0$ and $\upalpha$ does not end with $s_i$). By \eqref{inverse Si} $t_\upalpha^{-1}(S_i)\cong N$ for some finite length module $N$.
Hence by \ref{trivial}\eqref{trivial 2}
\[
[S_i,t_\upalpha b\kern 1pt]_{\geq d+1}=[N,b\kern 1pt]_{\geq d+1}=0.
\]

\medskip
\noindent
(b)  $S_i\in\cX_\upalpha$ (equivalently, $i\neq 0$ and $\upalpha$ ends with $s_i$). By \eqref{inverse Si} $t_\upalpha^{-1}(S_i)\cong N[1]$ for some finite length module $N$.  Hence again by \ref{trivial}\eqref{trivial 2}
\[
[S_i,t_\upalpha b\kern 1pt]_{\geq d+2}=[N[1],b\kern 1pt]_{\geq d+2}=[N,b\kern 1pt]_{\geq d+1}=0,
\]
and 
\[
[S_i,t_\upalpha b\kern 1pt]_{d+1}=[N,b\kern 1pt]_{d}\neq 0.
\]

Combining (a) and (b) proves \eqref{BT main 1}\eqref{BT main 2} in the case $k=1$.

\medskip
\noindent
\textbf{Induction Step.} We assume that the result is true for all paths with less than or equal to $k-1$ Deligne factors.  Write $\upalpha=\upalpha_k\circ\upbeta$ where $\upbeta\colonequals \upalpha_{k-1}\circ\hdots\circ\upalpha_1$.  By induction
\[
[\cS,t_{\upbeta} b\kern 1pt]_{\geq k+d}=0
\]
and $[S_j,t_{\upbeta} b\kern 1pt]_{k+d-1}\neq 0$ if and only if $j\neq 0$ and $\upalpha_{k-1}$ ends with $s_j$. Again there are only two cases:

\medskip
\noindent
(a) $S_i\in\cY_{\upalpha_k}$ (equivalently, $i=0$, or $i\neq 0$ and $\upalpha_k$ does not end with $s_i$). By \eqref{inverse Si} $t_{\upalpha_k}^{-1}(S_i)\cong N$ for some finite length module $N$. Hence
\[
[S_i,t_\upalpha b\kern 1pt]_{\geq k+d}
=
[t_{\upalpha_k}^{-1}S_i,t_\upbeta b\kern 1pt]_{\geq k+d}
=
[N,t_\upbeta b\kern 1pt]_{\geq k+d}
\stackrel{\scriptsize\mbox{\ref{trivial}\eqref{trivial 1}}}{=}
0.
\]

\medskip
\noindent
(b) $S_i\in\cX_{\upalpha_k}$ (equivalently, $i\neq 0$ and $\upalpha_k$ ends with $s_i$).  By \eqref{inverse Si} $t_{\upalpha_k}^{-1}(S_i)\cong N[1]$ for some finite length module $N$.  Thus
\[
[S_i,t_\upalpha b\kern 1pt]_{\geq k+d+1}\
=
[N[1],t_{\upbeta} b\kern 1pt]_{\geq k+d+1}
=
[N,t_\upbeta b\kern 1pt]_{\geq k+d}
\stackrel{\scriptsize\mbox{\ref{trivial}\eqref{trivial 1}}}{=}
0.
\]
Similarly
\[
[S_i,t_\upalpha b\kern 1pt]_{k+d}\
=
[N[1],t_{\upbeta} b\kern 1pt]_{k+d}
=
[N,t_\upbeta b\kern 1pt]_{k+d-1}
\]
so it remains to show that $[N,t_\upbeta b\kern 1pt]_{k+d-1}\neq 0$.  But by \ref{final cor}, there exists $j\neq 0$ such that $\upalpha_k$ starts with $s_j$, and $S_j\hookrightarrow N$.  Write $C$ for the cokernel, which necessarily has finite length, and consider the long exact sequence
\[
\hdots\to [C,t_\upbeta b\kern 1pt]_{k+d-1}\to [N,t_\upbeta b\kern 1pt]_{k+d-1}\to [S_j,t_\upbeta b\kern 1pt]_{k+d-1}\to [C,t_\upbeta b\kern 1pt]_{k+d}=0.
\]
Since $\upalpha_k$ starts with $s_j$,  necessarily $\upalpha_{k-1}$ ends with $s_j$, else $s_j\circ\alpha_{k-1}$ is an atom by \ref{standard Cox}, which would contradict the fact that  $\upalpha_k\circ\upalpha_{k-1}\circ\hdots\circ\upalpha_1$ is in Deligne normal form.  Thus $[S_j,t_\upbeta b\kern 1pt]_{k+d-1}\neq 0$ by the inductive hypothesis.  It follows that  $[N,t_\upbeta b\kern 1pt]_{k+d-1}\neq 0$.

Combining (a) and (b) proves \eqref{BT main 1}\eqref{BT main 2} in the case of $k$ factors, so by induction the result follows.
\end{proof}

The remainder of the proof of faithfulness is straightforward.
\begin{defin}
Define the groupoid $\cG_\Lambda$ as follows:
\begin{enumerate}
\item The vertices are $\Db(\mod \Lambda_C)$, for chambers $C$ of $\cH$.
\item The morphisms between any two vertices are all triangle equivalences between the corresponding derived categories.
\end{enumerate} 
\end{defin}
By \ref{simplicial and functors ok}\eqref{mut funct iso 2} and \ref{any atom will do}, there is a natural functor
\[
F_\Lambda\colon \mathds{G}_{\cH}\to\cG_\Lambda
\]
which sends a simple wall crossing $s_i$ to the corresponding equivalence $t_i$.

\begin{thm}\label{main groupoid faith}
The functor $F_\Lambda$ is faithful.
\end{thm}
\begin{proof}
This is an easy induction. We use \ref{faith to pos}, so suppose that 
\[
t_\upalpha=t_\upbeta\colon \Db(\mod\Lambda_C)\to\Db(\mod\Lambda_D)
\] 
for some $\upalpha,\upbeta\in \mathds{G}_{\cH}^+$.  Since $t_{\upalpha}=t_\upbeta$, we deduce from \ref{BT main}\eqref{BT main 3} that $\upalpha$ and $\upbeta$ have the same number of Deligne factors, so write
\[
\upalpha=\upalpha_k\hdots\upalpha_1\quad\mbox{and}\quad\upbeta=\upbeta_k\hdots\upbeta_1
\]  
in Deligne normal form.  By induction, it is enough to show that $\upalpha_k=\upbeta_k$ and $t_{\upalpha_{k-1}\hdots\upalpha_1}=t_{\upbeta_{k-1}\hdots\upbeta_1}$. We may assume that $\ell\colonequals \ell(\upalpha_k)\leq\ell(\upbeta_k)$. By \ref{BT main}\eqref{BT main 2}, since $t_\upalpha=t_\upbeta$, both $\upalpha_k$ and $\upbeta_k$ end with the same simple wall crossing, say $s_{i_1}$, so we can write $\upalpha_k=s_{i_1}\widetilde{\upalpha}_k$ and $\upbeta_k=s_{i_1}\widetilde{\upbeta}_k$.  
 Hence applying $t_{{i_1}}^{-1}$ to $t_\upalpha=t_\upbeta$ we deduce that $
t^{\phantom 1}_{\widetilde{\upalpha}_k\upalpha_{k-1}\hdots\upalpha_1}=t_{\widetilde{\upbeta}_k\upbeta_{k-1}\hdots\upbeta_1}$.

Repeating the above argument, we can write $\upalpha_k=s_{i_1}\hdots s_{i_\ell}$ and $\upbeta_k=s_{i_1}\hdots s_{i_\ell}\upgamma$ 
for some $\upgamma\in\mathds{G}_{\cH}^{+}$, and so we have  $
t_{\upalpha_{k-1}\hdots\upalpha_1}=t_{\upgamma\upbeta_{k-1}\hdots\upbeta_1}$.  But again by  \ref{BT main}\eqref{BT main 3}, $\upgamma$ must be a length zero path. Hence we have $\upalpha_k=s_{i_1}\hdots s_{i_\ell}=\upbeta_k$ and $t_{\upalpha_{k-1}\hdots\upalpha_1}=t_{\upbeta_{k-1}\hdots\upbeta_1}$, as required.
\end{proof}

\begin{cor}\label{main group faith}
For every chamber $C$, the induced map
\[
\fundgp(\mathbb{C}^n\backslash \cH_\mathbb{C})\to\Aut\Db(\mod\Lambda_C)
\]
 is an injective group homomorphism
\end{cor}
\begin{proof}
By \ref{group is faithful}, this follows immediately from \ref{main groupoid faith}.
\end{proof}

\subsection{Geometric Corollaries}\label{geo cor sect}
Although the above results were stated in the formal fibre setting, they easily imply the following global results.   

\begin{cor}\label{main flop result}
Suppose that $f\colon X\to X_{\con}$ is a flopping contraction between $3$-folds, where $X$ is smooth, and all curves in the contraction $f$ are individually floppable. Then there is an injective group homomorphism 
\[
\upvarphi\colon\fundgp(\mathbb{C}^n\backslash \cH_\mathbb{C})\to \Aut\Db(\coh X).
\]
\end{cor}
\begin{proof}
As in \cite[6.2]{DW3}, the functors in the image of $\upvarphi$ fix the skyscraper sheaves away from the flopping curves. Hence the relations can be detected on the formal fibre,  where the result is \ref{main group faith}.
\end{proof}

In the case when the $n$ curves are not individually floppable, there is still a group action, but only by a subgroup $S$ of $\fundgp(\mathbb{C}^n\backslash \cH_\mathbb{C})$ defined to be the subgroup generated by the $J$-twists of \cite{DW3}, where $J$ runs through all subsets of $\{1,\hdots,n\}$.  The proof of faithfulness extends to this case too.

\begin{cor}\label{main flop result arb curves}
Suppose that $f\colon X\to X_{\con}$ is a flopping contraction between $3$-folds, where $X$ is smooth. Then there is an injective group homomorphism 
\[
S\to \Aut\Db(\coh X).
\]
\end{cor}
\begin{proof}
Again, by \cite[6.2]{DW3}, the functors in the image of the above homomorphism fix the skyscraper sheaves away from the flopping curves.  Hence the relations can be detected on the formal fibre. Since there $\fundgp(\mathbb{C}^n\backslash \cH_\mathbb{C})$ acts faithfully by \ref{main group faith}, so does any subgroup.
\end{proof}

Recall that if $\cA$ is the heart of a bounded $t$-structure on a triangulated category $\cD$, and $\cA$ admits a torsion pair $(\cT,\cF)$, then the tilt of $\cA$ with respect to this torsion pair is defined to be
\[
\cA^{\sharp}\colonequals \{ E\in\cD\mid H^i(E)=0\mbox{ for }i\notin\{-1,0\}, H^{-1}(E)\in\cF\mbox{ and }H^0(E)\in\cT\}.
\]
By \cite[2.1]{HRS}, $\cA^{\sharp}$ is also the heart of a bounded $t$-structure on $\cD$.

Now for a $3$-fold flopping contraction $f\colon X\to X_{\rm con}$,  consider the full subcategories
\begin{align*}
\cT_{0}&\colonequals \{ T\in \coh X \mid {\bf R}^1f_*(T)=0\}\\
\cF_{0}&\colonequals \{ F\in\coh X \mid f_*(F)=0, {\rm Hom}(\cC,F)=0\},
\end{align*}
where $\cC\subset \coh X$ is the full subcategory consisting of objects $E$ such that ${\bf R} f_*(E)=0$.
Then $(\cT_{0},\cF_{0})$ is a torsion pair by \cite[Lemma 3.1.2]{VdB1d}, and the category of perverse sheaves relative to $f$ is defined to be
\[
\Per(X,X_{\rm con})\colonequals (\coh X)^{\sharp},
\] 
namely the tilt of the standard heart $\coh X\subset \Db(\coh X)$ with respect to the torsion pair $(\cT_{0},\cF_{0})$.

The following is a further consequence of the results in this paper, and may be of independent interest.  The first part is implicit in \cite{DW3}, the second part is new.

\begin{thm}\label{single tilt perverse}
Consider two crepant resolutions 
\[
\begin{tikzpicture}
\node (X) at (-1,1) {$X$};
\node (Y) at (1,1) {$Y$};
\node (R) at (0,0) {$\Spec R$};
\draw[->] (X)--(R);
\draw[->] (Y)--(R);
\end{tikzpicture}
\]
of $\Spec R$, where $R$ is an isolated cDV singularity.
\begin{enumerate}
\item Given two minimal chains of flops connecting $X$ and $Y$, the composition of flop functors associated to each chain are functorially isomorphic.
\item Perverse sheaves on $Y$, namely $\Per(Y,R)$, can be obtained from perverse sheaves on $X$, namely $\Per(X,R)$, by a single tilt at a torsion pair.
\end{enumerate}\end{thm}
\begin{proof}
(1) By \cite[1.10, 1.12]{Deligne}, any two minimal paths can be identified provided that in the Deligne groupoid the codimension two relations hold.  By \cite[3.20]{DW3} the codimension two relations are precisely correspond to the braiding of the $2$-curve flop functors, which is proved in \cite[3.9, 3.20]{DW3}.\\
(2) Consider a minimal path of flops
\[
\Db(\coh X)
\xrightarrow{\flop_{i_1}}
\Db(\coh X_{i_1})
\xrightarrow{\flop_{i_2}}
\hdots
\xrightarrow{\flop_{i_n}}
\Db(\coh Y)
\]
connecting $X$ and $Y$.  By \cite{VdB1d} $X$ is derived equivalent to $\End_R(M)$ say, and $Y$ is derived equivalent to $\End_R(N)$ say, and under this identification $\Per(X,R)$ corresponds to $\mod\End_R(M)$, and $\Per(Y,R)$ corresponds to $\mod\End_R(N)$. Hence it suffices to show that $\mod\End_R(N)$ can be obtained from $\mod\End_R(M)$ by a tilt at a torsion pair.

Consider $T\colonequals \Hom_R(M,N)$.  This is a tilting $\End_R(M)$-module, by \cite[4.17]{IW4}.   But since $\End_R(M)$ is noetherian,
\begin{align*}
\cT&\colonequals \{ X\in\mod \End_R(M)\mid \Ext^1_{\End_R(M)}(T,X)=0\}\\
\cF&\colonequals \{ X\in\mod \End_R(M)\mid \Hom_{\End_R(M)}(T,X)=0\}
\end{align*} 
gives a torsion pair $(\cT,\cF$) on $\mod\End_R(M)$; the proof is identical to \cite[2.7(3)]{SY}. Using the finitely-generated version of the equivalences \eqref{TF and XY}, it is then clear that $\mod\End_R(N)$ is obtained from $\mod\End_R(M)$ by tilting at $(\cT,\cF)$.

\end{proof}

\appendix
\section{Brav--Thomas Revisited}\label{BT appendix}

In this appendix, which can be read independently of the previous sections, we give a direct proof of the faithfulness of the braid action on the minimal resolution of Kleinian singularities, just to demonstrate that our torsion pairs viewpoint simplifies the \cite{BT} proof.  Thus  in this section we consider the minimal resolution $X\to\Spec R$ of a Kleinian singularity, let $\Lambda$ denote the completion of the preprojective algebra of the corresponding extended Dynkin diagram, and set $b\colonequals \cS$, where $\cS$ is the direct sum of the vertex simples $S_0,S_1,\hdots,S_n$.

The initial step requires the following elementary lemma, which replaces \ref{trivial}.

\begin{lemma}\label{trivial appendix}
Suppose that $M\in\fl\Lambda$.
\begin{enumerate}
\item\label{trivial 1 B} If $y\in\Db(\mod\Lambda)$ is such that $[\cS,y\kern 1pt]_{\geq p}=0$, then $[M,y\kern 1pt]_{\geq p}=0$. 
\item\label{trivial 2 B} $[M,\cS]_{2}\neq 0$ and $[M,\cS]_{\geq 3}=0$.
\end{enumerate}
\end{lemma}
\begin{proof}
(1) is an easy induction on the length of the filtration of $M$, using the long exact sequence from $[-,y\kern 1pt]$.\\
(2) Since $\Lambda$ is $2$-CY, $[\cS,\cS]_{\geq 3}=0$, so the second statement is a consequence of \eqref{trivial 1 B}.  The first also follows  by an induction on the length of the filtration of $M$, using the fact that $[S_i,\cS]_2\neq 0$ for all $0\leq i\leq n$.
\end{proof}

For every primitive idempotent $e_i$ corresponding to a vertex of the extended Dynkin diagram, following \cite[\S6]{IR} we set
\[
I_i\colonequals \Lambda(1-e_i)\Lambda.
\] 
It is known by \cite[Section 6]{DW1} that $\uprhom_\Lambda(I_i,-)$ is functorially isomorphic to the twist functor $t_i$.  To control iterations, for any $\upalpha\in W$ where $W$ is the associated Weyl group, choose a reduced expression $\upalpha=s_{i_n}\circ\hdots\circ s_{i_1}$ and define
\[
I_\upalpha\colonequals I_{i_n}\hdots I_{i_1}.
\]
Since the expression is reduced, 
\[
I_\upalpha\cong I_{i_n}\otimes^{\bf L}_{\Lambda}\hdots\otimes^{\bf L}_{\Lambda} I_{i_1}
\]
by \cite[2.21]{SY}, so that
\begin{equation}
t_\upalpha\colonequals \uprhom_{\Lambda}(I_{\upalpha},-)\cong t_{i_n}\circ\hdots\circ t_{i_1}.\label{3B2}
\end{equation}
By the usual torsion pair associated to a tilting module, as in \cite[2.9]{SY} and \S\ref{tracking torsion pairs}, for any vertex simple $S_i$,  either $S_i\in\cX_\upalpha$ or $S_i\in\cY_\upalpha$, where
\begin{align*}
\cX_\upalpha&\colonequals \{ N\in\fl \Lambda\mid N\otimes_{\Lambda} I_\upalpha=0\}\\
\cY_\upalpha&\colonequals \{ N\in\fl \Lambda\mid \Tor_1^{\Lambda}(N,I_\upalpha)=0\},
\end{align*} 
and furthermore the equivalence \eqref{3B2} forces
\begin{eqnarray}
t_\upalpha^{-1}(S_i)=
\left\{
\begin{array}{cl}
\Tor_1(S_i,I_\upalpha)[1]&\mbox{if }S_i\in\cX_\upalpha\\
S_i\otimes I_\upalpha&\mbox{if }S_i\in\cY_\upalpha.
\end{array}
\right.\label{inverse Si2}
\end{eqnarray}

There is a corresponding version of the results \ref{Key SY lemma}, \ref{S0 ok} and \ref{final cor}, which we will use freely below, since these were already very well known \cite[2.28, 5.4]{SY} in the preprojective algebra setting.  With this, we can now prove the main technical lemma \cite[Prop.\ 3.1]{BT} in the setting of minimal resolutions of Kleinian singularities. 

\begin{prop}\label{BT main B}
Let $1\neq \upalpha\in \mathds{G}^+_{\cH}$ have Deligne normal form $\upalpha=\upalpha_k\circ\hdots\circ \upalpha_1$.  Then
\begin{enumerate}
\item\label{BT main 1 B} $[\cS,t_\upalpha\cS]_{\geq k+3}=0$.
\item\label{BT main 2 B} $[S_i,t_\upalpha\cS]_{k+2}\neq 0$ if and only if $i\neq 0$ and the atom $\upalpha_k$ ends (up to the relations in $\mathds{G}^+_{\cH}$) by passing through wall $i$.  In particular $[\cS,t_\upalpha\cS]_{k+2}\neq 0$.
\item\label{BT main 3 B} The maximal $p$ such that $[\cS,t_\upalpha\cS]_p\neq 0$ is precisely $p=k+2$.
\end{enumerate}
\end{prop}
\begin{proof}
Statement \eqref{BT main 3 B} follows immediately from \eqref{BT main 1 B} and \eqref{BT main 2 B}, so we prove both \eqref{BT main 1 B} and \eqref{BT main 2 B} together using induction on the number of Deligne factors.

\medskip
\noindent
\textbf{Base Case: $k=1$}, i.e.\ $\upalpha$ is an atom.  Since $S_i$ is simple, there are only two cases, namely  $S_i\in\cY_\upalpha$ or $S_i\in\cX_\upalpha$, and using \ref{Key SY lemma} and \ref{S0 ok} we can characterise these:

\medskip
\noindent
(a) $S_i\in\cY_\upalpha$ (equivalently, $i=0$, or $i\neq 0$ and $\upalpha$ does not end with $s_i$). By \eqref{inverse Si2} $t_\upalpha^{-1}(S_i)\cong M$ for some finite length module $M$.
Hence by \ref{trivial appendix}\eqref{trivial 2}
\[
[S_i,t_\upalpha\cS]_{\geq 3}=[M,\cS]_{\geq 3}=0.
\]

\medskip
\noindent
(b)  $S_i\in\cX_\upalpha$ (equivalently, $i\neq 0$ and $\upalpha$ ends with $s_i$). By \eqref{inverse Si2} $t_\upalpha^{-1}(S_i)\cong M[1]$ for some finite length module $M$.  Hence again by \ref{trivial appendix}\eqref{trivial 2}
\[
[S_i,t_\upalpha\cS]_{\geq 4}=[M[1],\cS]_{\geq 4}=[M,\cS]_{\geq 3}=0,
\]
and 
\[
[S_i,t_\upalpha\cS]_{3}=[M,\cS]_{2}\neq 0.
\]

Combining (a) and (b) proves \eqref{BT main 1 B}\eqref{BT main 2 B} in the case $k=1$.

\medskip
\noindent
\textbf{Induction Step.} We assume that the result is true for all paths with less than or equal to $k-1$ Deligne factors.  Write $\upalpha=\upalpha_k\circ\upbeta$ where $\upbeta\colonequals \upalpha_{k-1}\circ\hdots\circ\upalpha_1$.  By induction
\[
[\cS,t_{\upbeta}\cS]_{\geq k+2}=0
\]
and $[S_j,t_{\upbeta}\cS]_{k+1}\neq 0$ if and only if $j\neq 0$ and $\upalpha_{k-1}$ ends with $s_j$. Again there are only two cases:

\medskip
\noindent
(a)  $S_i\in\cY_{\upalpha_k}$ (equivalently, $i=0$, or $i\neq 0$ and $\upalpha_k$ does not end with $s_i$). By \eqref{inverse Si2} $t_{\upalpha_k}^{-1}(S_i)\cong M$ for some finite length module $M$. Hence
\[
[S_i,t_\upalpha\cS]_{\geq k+2}
=
[t_{\upalpha_k}^{-1}S_i,t_\upbeta\cS]_{\geq k+2}
=
[M,t_\upbeta\cS]_{\geq k+2}
\stackrel{\scriptsize\mbox{\ref{trivial appendix}\eqref{trivial 1}}}{=}
0.
\]

\medskip
\noindent
(b) $S_i\in\cX_{\upalpha_k}$ (equivalently, $i\neq 0$ and $\upalpha_k$ ends with $s_i$).  By \eqref{inverse Si2} $t_{\upalpha_k}^{-1}(S_i)\cong M[1]$ for some finite length module $M$.  Thus
\[
[S_i,t_\upalpha\cS]_{\geq k+3}\
=
[M[1],t_{\upbeta}\cS]_{\geq k+3}
=
[M,t_\upbeta\cS]_{\geq k+2}
\stackrel{\scriptsize\mbox{\ref{trivial appendix}\eqref{trivial 1}}}{=}
0.
\]
Similarly
\[
[S_i,t_\upalpha\cS]_{k+2}\
=
[M[1],t_{\upbeta}\cS]_{k+2}
=
[M,t_\upbeta\cS]_{k+1}
\]
so it remains to show that $[M,t_\upbeta\cS]_{k+1}\neq 0$.  But by \ref{final cor}, there exists $j\neq 0$ such that $\upalpha_k$ starts at $s_j$, and $S_j\hookrightarrow M$.  Write $C$ for the cokernel, which necessarily has finite length, and consider the long exact sequence
\[
\hdots\to [C,t_\upbeta\cS]_{k+1}\to [M,t_\upbeta\cS]_{k+1}\to [S_j,t_\upbeta\cS]_{k+1}\to [C,t_\upbeta\cS]_{k+2}=0.
\]
Since $\upalpha_k$ starts with $s_j$,  necessarily $\upalpha_{k-1}$ ends with $s_j$, else $s_j\circ\alpha_{k-1}$ is an atom by \ref{standard Cox}, which would contradict the fact that  $\upalpha_k\circ\upalpha_{k-1}\circ\hdots\circ\upalpha_1$ is in Deligne normal form.  Thus $[S_j,t_\upbeta\cS]_{k+1}\neq 0$ by the inductive hypothesis.  It follows that  $[M,t_\upbeta\cS]_{k+1}\neq 0$.

Combining (a) and (b) proves \eqref{BT main 1 B}\eqref{BT main 2 B} in the case of $k$ factors, so by induction the result follows.
\end{proof}

From here, the proof of faithfulness follows exactly as in \cite[Thm.\ 3.1]{BT}.  Alternatively, we can use \ref{faith to pos} as in \ref{main groupoid faith} to deduce that the groupoid action is faithful. Since $B_\Gamma\cong \fundgp( (\mathbb{C}^n\backslash \cH_\mathbb{C})/W_\Gamma)$, and each vertex of $\cG$ is by definition the same $\Db(\coh X)$,  as is standard by identifying all vertices we can simply re-interpret the faithful groupoid action as an injective group homomorphism $B_\Gamma\to\Aut\Db(\coh X)$.

\section{Tilting Background}\label{tilting appendix}

In this appendix, which is logically independent of all other sections, we give some known tilting results that were used in the text, and we also prove \ref{easy one way appA text} and \ref{DIJ lemmaB}.

 Throughout $\Lambda$ is a basic $\mathfrak{R}$-algebra, where $\mathfrak{R}$ is a complete local domain. Recall that if $T\in\tilt_0\Lambda$ and its mutation $\upnu_iT$ at a direct summand $T_i$ exists, either there is an exact sequence 
\[
0\to T_i\xrightarrow{f}T'\to U_i\to 0
\]
where $f$ is a minimal left $\add(T/T_i)$-approximation, or an exact sequence
\[
0\to U_i\to T'\xrightarrow{g}T_i\to 0
\]
where $g$ is a minimal right $\add (T/T_i)$-approximation.
By definition, $T>\upnu_iT$ in the former case, and $T<\upnu_iT$ in the latter case.

Suppose that $T\in\tilt \Lambda$ with $\End_\Lambda(T)\cong\Gamma$.  By projectivization, the indecomposable summands of $\Gamma$ correspond to the indecomposable summands of $T$.  Hence we can try to mutate $T\in\tilt\Lambda$ to form $\upnu_iT$, and similarly we can try to mutate $\Gamma\in\tilt\Gamma$ to form $\upnu_i\Gamma$.   Although the following is elementary and is known to experts, references to the literature only exist when $\mod\Lambda$ is Hom-finite, so here we give the proof in full.  

\begin{prop}\label{flow big to small}
Suppose that $T\in\tilt\Lambda$, and set $\Gamma\colonequals \End_\Lambda(T)$.  If $\upnu_i T$ exists and further $T>\upnu_i T$, then $\upnu_i\Gamma\in\tilt\Gamma$ exists, there is an isomorphism  $\upnu_i T\cong \upnu_i\Gamma\otimes^{\bf L}_\Gamma T$ in $\Db(\mod\Lambda)$, and further $\upnu_iT\in\cT\colonequals \{ N\in\mod \Lambda\mid \Ext^1_{\Lambda}(T,N)=0\}$.
\end{prop}
\begin{proof}
To ease notation write $V\colonequals T/T_i$. \\
(1) Since $\upnu_iT=V\oplus U_i$ exists and $T>\upnu_iT$, as above there exists an exact sequence
\[
0\to T_i\xrightarrow{f}T'\to U_i\to 0
\]
where $f$ is a minimal left $\add V$-approximation.  Applying $\Hom_\Lambda(T,-)$ gives an exact sequence
\begin{equation}
0\to \Hom_\Lambda(T,T_i)\xrightarrow{f\circ } \Hom_\Lambda(T,T')\to \Hom_\Lambda(T,U_i)\to 0\label{proj res Ui}
\end{equation}
Write $\Gamma=\Hom_{\Lambda}(T,T)=\Hom_\Lambda(T,V)\oplus \Hom_\Lambda(T,T_i)\colonequals \Gamma_V\oplus\Gamma_i$, then by projectivisation \eqref{proj res Ui} is a projective resolution of $\Hom_\Lambda(T,U_i)$.  We claim that $(f\circ)$ is a minimal left $\add\Gamma_V$-approximation.  To see this, simply apply $\Hom_\Gamma(-,\Gamma_V)$ to \eqref{proj res Ui} to obtain a commutative diagram
\[
\begin{tikzpicture}
\node (A1) at (0,0) {$\Hom_\Gamma(\Gamma_i,\Gamma_V)$};
\node (A2) at (4,0) {$\Hom_\Gamma(\Hom_{\Lambda}(T,T'),\Gamma_V)$};
\node (B1) at (0,-1.5) {$\Hom_\Lambda(T_i,V)$};
\node (B2) at (4,-1.5) {$\Hom_\Lambda(T',V)$};

\draw[->] (A2) -- (A1);
\draw[->>] (B2) -- node[above]{$\scriptstyle \circ f$} (B1);
\draw[<-] (B2) -- node[right] {$\scriptstyle \sim$} (A2);
\draw[<-] (B1) --  node[right] {$\scriptstyle \sim$} (A1);
\end{tikzpicture}
\]
where the vertical maps are isomorphisms by projectivisation, and the bottom map is surjective since $f$ is an $\add V$-approximation.  It follows that the top map is surjective, and hence $(f\circ)$ is a left $\add\Gamma_V$-approximation.  The minimality of $(f\circ)$ follows since the left $\add V$-approximation $f$ is minimal, and the functor $\Hom_{\Lambda}(T,-)\colon \add T\to\proj\Gamma$ is fully faithful.

As is standard \cite[6.6]{IW4}, since $(f\circ )$ in \eqref{proj res Ui} is injective and an approximation, it follows that $\Gamma_V\oplus\Hom_\Lambda(T,U_i)\in\tilt\Gamma$, and evidently $\upnu_i\Gamma\cong \Gamma_V\oplus\Hom_\Lambda(T,U_i)$ since $\upnu_i\Gamma$ and $\Gamma$ differ at only one indecomposable summand.

Now, using \eqref{proj res Ui} to compute the derived tensor in $\Db(\mod\Lambda)$, observe first that 
\begin{align*}
\Hom_\Lambda(T,U_i)\otimes^{\bf L}_\Gamma T&\cong
\hdots \to 0\to \Hom_\Lambda(T,T_i)\otimes_\Gamma T
\to \Hom_\Lambda(T,T')\otimes_\Gamma T\to 0\to\hdots\\
&\cong
\hdots \to 0\to T_i\xrightarrow{f} T'\to 0\to\hdots,
\end{align*}
which since $f$ is injective, is clearly isomorphism to $U_i$ (in degree zero).  Hence
\begin{align*}
\upnu_i\Gamma\otimes^{\bf L}_\Gamma T
&\cong(\Hom_\Lambda(T,V)\otimes_\Gamma T)\oplus (\Hom_\Lambda(T,U_i)\otimes^{\bf L}_\Gamma T)\\
&\cong V\oplus U_i,
\end{align*}
where $\Hom_\Lambda(T,V)\otimes_\Gamma T\cong V$ holds since $T$ is tilting and $V$ is projective. It follows that $\upnu_i\Gamma\otimes^{\bf L}_\Gamma T\cong \upnu_iT$ in $\Db(\mod\Lambda)$.  Applying $\RHom_\Lambda(T,-)$ gives the final statement.
\end{proof}

\begin{lemma}[\ref{easy one way appA text}]\label{easy one way appA}
Suppose that $\Lambda$ is a basic $\mathfrak{R}$-algebra, where $\mathfrak{R}$ is a complete local domain.  If $T,U\in\tilt_0\Lambda$ are related by a mutation at an indecomposable summand, then $C_T$ and $C_U$ do not overlap, and are separated by a codimension one wall.
\end{lemma}
\begin{proof}
By assumption, there are indecomposable modules $T_0,\hdots,T_n$ and $U_n$ such that $T=T_{<n}\oplus T_n$ and $U=T_{<n}\oplus U_n$, where $T_{<n}\colonequals \bigoplus_{i=0}^{n-1}T_i$.   We may assume that $T>U$, and then there is an exact sequence 
\[
0\rightarrow T_n\rightarrow X_{<n}\rightarrow U_n\rightarrow 0
\] 
where $X_{<n}\in\add T_{<n}$, say $X_{<n}\colonequals T_0^{\oplus a_0}\oplus T_1^{\oplus a_1}\oplus \hdots\oplus T_{n-1}^{\oplus a_{n-1}}$.
Thus, recalling that the $[-]$ notation works modulo $\Span\{\mathbf{e}_0\}$, we see that
\[
[T_n]=-[U_n]+\sum_{i=1}^{n-1} a_i[T_i] 
\]  
in $\Uptheta_{\Lambda}$, and so
\begin{align*}
C_T&\colonequals \{ \sum_{i=1}^n\upvartheta_i [T_i]\, \mid \upvartheta_i>0\mbox{ for all }1\leq i\leq n\}\\
&=\{  \sum_{i=1}^{n-1}(\upvartheta_i+a_i\upvartheta_n) [T_i]-\upvartheta_n[U_n] \, \mid \upvartheta_i>0\mbox{ for all }1\leq i\leq n\}.
\end{align*}
Since $U$ is tilting, the classes of indecomposable summands of $U$, namely $[T_0],\hdots, [T_{n-1}]$, and $[U_n]$, span $K_0\otimes _{\mathbb{Z}}\mathbb{R}\cong \mathbb{R}^{n+1}$. Hence they form a basis of $K_0\otimes _{\mathbb{Z}}\mathbb{R}$, and in particular the classes $[T_1],\hdots,[T_{n-1}],[U_n]$ in $\Uptheta_{\Lambda}$ form a basis of $\Uptheta_{\Lambda}$.

Write $H\subset \Uptheta_{\Lambda}$ for the linear subspace spanned by $[T_1],\hdots,[T_{n-1}]$. Then $H$ separates $\Uptheta_{\Lambda}$ into two half spaces $H_{+}\colonequals \{\sum_{i=1}^{n-1}b_i[T_i]+a[U_n] \mid a>0\}$  and $H_{-}\colonequals \{\sum_{i=1}^{n-1}b_i[T_i]+a[U_n] \mid a<0\}$. Since  $C_T\subset H_{-}$, $C_U\subset H_{+}$, and $H_{+}\cap H_{-}=\emptyset$, we obtain $C_T\cap C_U=\emptyset$. It is clear that $C_T$ and $C_U$ are separated by a codimension one wall contained in $H$.
\end{proof}

\begin{lemma}\label{DIJ lemmaA}
Suppose that $\Lambda$ is a basic $\mathfrak{R}$-algebra, where $\mathfrak{R}$ is a complete local domain.  Suppose that $T,U\in\tilt_0\Lambda$ are related  by a mutation at an indecomposable summand $T_n$.  If $T>U$, then  there exists an exact sequence $0\rightarrow \Lambda\rightarrow T' \rightarrow T'' \rightarrow 0$ such that  $T_n\notin\add T''$.
\end{lemma}
\begin{proof}
By the definition of tilting modules, there is an exact sequence  $0\rightarrow \Lambda\rightarrow T' \rightarrow T'' \rightarrow 0$ with $T',T''\in \add T$, and this induces the following triangle in ${\rm D^b}(\mod \Lambda)$
\[
\Lambda\rightarrow T' \rightarrow T'' \xrightarrow{f}\Lambda[1].
\]
Since $\Ext^1(T,T)=0$, we  see that $f\colon T'' \rightarrow \Lambda[1]$ is a right $\add T$-approximation. Replacing $T'$ and $T''$ if necessary, we may assume that the approximation $f$ is right minimal, and we will show that for such a sequence $0\rightarrow \Lambda\rightarrow T' \rightarrow T'' \rightarrow 0$, we have $T_n\notin \add T''$.

Suppose that $T_n\in \add T''$, and let $Y$ be the summand of $T''$ such that  $T_n\notin \add Y$ and  $T''=(T_n)^{\oplus a}\oplus Y$ for some $a>0$.  Let $f_n\colon (T_n)^{\oplus a}\rightarrow \Lambda[1]$ and $f_Y\colon Y\rightarrow \Lambda[1]$ be the components of $f$. By assumption, there are indecomposable modules $T_n$ and $U_n$ such that  $T=X\oplus T_n$ and $U=X\oplus U_n$. Since $T>U$, there is an exact sequence 
\[
0\rightarrow T_n\xrightarrow{g} X' \rightarrow U_n \rightarrow0,
\]
where $X'\in\add X$. Applying $\Hom_{\Db(\mod\Lambda)}(-,\Lambda[1])$ to the above gives an exact sequence
\[
\Hom_{\Db(\mod\Lambda)}(X',\Lambda[1])\xrightarrow{\circ g}
\Hom_{\Db(\mod\Lambda)}(T_n,\Lambda[1])\rightarrow
\Hom_{\Db(\mod\Lambda)}(U_n[-1],\Lambda[1])=0,
\]
since $\pd_\Lambda U_n\leq 1$.
Hence there exists a morphism $h\colon X'^{\oplus a}\rightarrow \Lambda[1]$ such that  $f_n=h\circ g^{\oplus a}$.
But $h+f_Y\colon X'^{\oplus a}\oplus Y\rightarrow \Lambda[1]$ is a right $\add T$-approximation, with $X'^{\oplus a}\oplus Y\in \add X$, and so $T_n\notin \add(X'^{\oplus a}\oplus Y)$. This contradicts the minimality of $f\colon T''\rightarrow \Lambda[1]$, since $T_n\in \add T''$. Hence $T_n\notin \add T''$. 
\end{proof}

\begin{thm}[\ref{DIJ lemmaB}]\label{DIJ lemmaB app}
Suppose that $\Lambda$ is a basic $\mathfrak{R}$-algebra, where $\mathfrak{R}$ is a complete local domain.  Suppose that $T,U\in\tilt_0\Lambda$ are related  by a mutation at an indecomposable summand, so by \ref{easy one way appA} $C_T$ and $C_U$ are separated by $H$.  Suppose that $[\Lambda]\notin H$.  Then $T>U$ iff $C_T$ lies on the same side of $H$ as $[\Lambda]$.
\end{thm}
\begin{proof}
($\Rightarrow$) Suppose that $T>U$. Since the summands of $T$ (excluding $T_0=P_0$) form a basis for $\Uptheta_\Lambda$, we can write
\[
[\Lambda]=b_1[T_1]+\hdots+ b_{n-1}[T_{n-1}] + b_n[T_n]
\]
Certainly $b_n\neq 0$, else $[\Lambda]\in H$, which is false by assumption.   Since by \ref{DIJ lemmaA} there are objects $T',T''\in \add T$ such that $[\Lambda]=[T']-[T'']$ and $T_n\notin \add T''$, necessarily $b_n>0$ given that it is non-zero.  It follows that $[\Lambda]$ is on the same side of $H$ as $C_T$.\\
($\Leftarrow$) If $\neg(T>U)$, then since by the assumption $T$ and $U$ are the mutation of each other at an indecomposable summand, necessarily $U>T$.  Replicating the above proof word-for-word, we conclude that $C_U$ is on the same side of $H$ as $[\Lambda]$.  Since $C_T$ is on the other side of $H$ than $C_U$, it follows that $C_T$ is not on the same side of $H$ as $[\Lambda]$.
\end{proof}


\begin{thebibliography}{DW3}



\bibitem[BT]{BT}
C.~Brav and H.~Thomas, \emph{Braid groups and Kleinian singularities}, Math.\ Ann.\ \textbf{351} (2011), no.~4, 1005--1017.

\bibitem[BB]{BBtilting}
S.~Brenner and M.~C.~R.~Butler, \emph{A spectral sequence analysis of classical tilting functors}, Handbook of tilting theory, 31--48, London Math.\ Soc.\ Lecture Note Ser., \textbf{332}, Cambridge Univ.\ Press, Cambridge, 2007.


\bibitem[BO]{BO}
A.~Bondal and D.~Orlov, \emph{Reconstruction of a variety from the derived category and groups of autoequivalences},  Compos.\ Math.\
\textbf{125} (2001), no.~3, 327--344.

\bibitem[CE]{CE}
H.~Cartan and S.~Eilenberg, \emph{Homological algebra}, Princeton Landmarks in Mathematics. Princeton University Press, Princeton, NJ, 1999.


\bibitem[CS]{CS}
I.~Cheltsov and C.~Shramov, \emph{Cremona groups and the icosahedron}, Monographs and Research Notes in Mathematics. CRC Press, Boca Raton, FL, 2016. xxi+504 pp.

\bibitem[D]{Deligne}
P. Deligne, \emph{Les immeubles des groupes de tresses g\'en\'eralis\'es}, Invent.\ Math.\ \textbf{17} (1972), 273--302.

\bibitem[DIJ]{DIJ}
L.~Demonet, O.~Iyama and  G.~Jasso, \emph{$\tau$-tilting finite algebras, $g$-vectors and brick--$\tau$-rigid correspondence}, {\sf arXiv:1503.00285}.

%\bibitem[DS]{DonSeg2}
%W.~Donovan and E.~Segal, \emph{Mixed braid group actions from deformations of surface singularities}, Comm.\ Math.\ Phys.\ (1) \textbf{335} (2014), 497--543.

\bibitem[DW1]{DW1}
W.~Donovan and M.~Wemyss, \emph{Noncommutative deformations and flops}, Duke Math.\ J.\ \textbf{165}, (2016), no.~8, 1397--1474. 

\bibitem[DW3]{DW3}
W.~Donovan and M.~Wemyss, \emph{Twists and braids for general 3-fold flops}, to appear JEMS, \textsf{arXiv:1504.05320}.

\bibitem[HRS]{HRS}
D.~Happel, I.~Reiten and S.~O.~Smal\o, \emph{Tilting in abelian categories and quasitilted algebras}, Mem.\ Amer.\ Math.\ Soc.\ \textbf{120} (1996), no.~575.

\bibitem[HU]{HU} D.~Happel and L.~Unger, \emph{On a partial order of tilting modules}, Algebr.\ Represent.\ Theory \textbf{8} (2005), no.~2, 147--156. 

\bibitem[H]{Hille}
L.~Hille, \emph{On the volume of a tilting module},  
Abh.\ Math.\ Sem.\ Univ. Hamburg \textbf{76} (2006), 261--277. 

\bibitem[IR]{IR}
O.~Iyama and I.~Reiten, \emph{Fomin-Zelevinsky mutation and tilting modules over Calabi-Yau algebras}, Amer.\ J.\ Math.\ \textbf{130} (2008), no.\ 4, 1087--1149.

\bibitem[IW]{IW4}
O.~Iyama and M.~Wemyss, \emph{Maximal modifications and Auslander--Reiten duality for non-isolated singularities}, Invent.\ Math.\ \textbf{197} (2014), no.~3, 521--586.

\bibitem[IW2]{IW6}
O.~Iyama and M.~Wemyss, \emph{Reduction of triangulated categories and Maximal Modification Algebras for $cA_n$ singularities}, to appear Crelle, {\sf arXiv:1304.5259}.

\bibitem[IW3]{IW9}
O.~Iyama and M.~Wemyss, \emph{Affine actions on $3$-folds via contracted preprojective algebras and Tits cone intersections}, in preparation.

%\bibitem[KS]{KS} M.~Kashiwara and P.~Schapira, \emph{Sheaves on manifolds. With a chapter in French by Christian Houzel}, Grundlehren  der Mathematischen Wissenschaften [Fundamental Principles of Mathematical Sciences], \textbf{292}. Springer-Verlag, Berlin, (1990).

\bibitem[K]{Katz}
S.~Katz, \emph{Small resolutions of Gorenstein threefold singularities}, Algebraic geometry: Sundance 1988, 61--70, Contemp.\ Math., \textbf{116}, Amer.\ Math.\ Soc., Providence, RI, 1991.


\bibitem[P1]{Paris}
L.~Paris, \emph{Universal cover of Salvetti's complex and topology of simplicial arrangements of hyperplanes}, Trans.\ Amer.\ Math.\ Soc.\ \textbf{340} (1993), no.~1, 149--178.


\bibitem[P2]{Paris2}
L.~Paris, \emph{On the fundamental group of the complement of a complex hyperplane arrangement}. Arrangements---Tokyo 1998, 257--272, Adv.\ Stud.\ Pure Math., \textbf{27}, Kinokuniya, Tokyo, 2000. 

\bibitem[P3]{Paris3}
L.~Paris, \emph{The covers of a complexified real arrangement of hyperplanes and their fundamental groups},  Topology Appl. \textbf{53} (1993), no.~1, 75--103.

\bibitem[P4]{Pinkham}
H.~Pinkham, \emph{Factorization of birational maps in dimension 3}, Singularities (P. Orlik, ed.), Proc.\ Symp.\ Pure Math., vol.\ 40, Part 2, 343--371, American Mathematical Society, Providence, 1983.

\bibitem[P5]{Priyavrat}
P.~Deshpande, \emph{Arrangements of submanifolds and the tangent bundle complement} (2011). Electronic Thesis and Dissertation Repository, \textbf{154}. \url{http://ir.lib.uwo.ca/etd/154}.

\bibitem[RS]{RS} C.~Riedtmann and  A.~Schofield, \emph{On a simplicial complex associated with tilting modules}, Comment.\ Math.\ Helv.\ \textbf{66} (1991), no.~1, 70--78.

\bibitem[Sa]{Salvetti}
M.~Salvetti, \emph{Topology of the complement of real hyperplanes in $\mathbb{C}^N$}, Invent.\ Math.\ \textbf{88} (1987), no.~3, 603--618.

\bibitem[ST]{ST}
P.~Seidel and R.~P.~Thomas, \emph{Braid group actions on derived categories of sheaves}, Duke Math.\ J.\ \textbf{108} (2001), 37--108.

\bibitem[SY]{SY}
Y.~Sekiya and K.~Yamaura, \emph{Tilting theoretical approach to moduli spaces over preprojective algebras}, Algebr.\ Represent.\ Theory \textbf{16} (2013), no.~6, 1733--1786.

\bibitem[Sw]{Swan}
R.~G.~Swan, \emph{Induced representations and projective modules},
Ann.\ of Math., \textbf{71} (1960),  no.~2,  552--578. 

\bibitem[V]{VdB1d}
M.~Van den Bergh, \emph{Three-dimensional flops and noncommutative rings}, 
Duke Math.\ J.\ \textbf{122} (2004), no.~3, 423--455. 

\bibitem[W]{HomMMP}
M.~Wemyss, \emph{Flops and clusters in the homological minimal model programme}, to appear Invent.\ Math., {\sf arXiv:1411.7189}.

\end{thebibliography}
\end{document}